\documentclass[a4paper, reqno, dvipdfmx]{amsart}

\usepackage{amsmath}
\usepackage{amssymb}
\usepackage{amsthm}
\usepackage[dvipdfmx]{graphicx}
\usepackage{amsfonts}
\usepackage{bm}
\usepackage{url}

\usepackage{amsmath,amsthm,amsfonts,amssymb}
\usepackage{graphicx} 
\usepackage{epstopdf}
\usepackage{enumerate}
\usepackage{url}
\usepackage{bm}
\usepackage{epsfig}
\usepackage{color}
\theoremstyle{definition}

  \newtheorem{thm}{Theorem}[section]
  
  \newtheorem{lem}[thm]{Lemma}
  
  \newtheorem{prop}[thm]{Proposition}

  \newtheorem{defi}[thm]{Definition}

\theoremstyle{remark}
\newtheorem{rem}{Remark}[section]

\numberwithin{equation}{section}\numberwithin{figure}{section}

\def\>{{\geq }}
\def\<{{\leq }}

\newcommand{\Res}{\mathop{\rm Res}}

\title[Witten--Reshetikhin--Turaev function for Seifert loops]
{Witten--Reshetikhin--Turaev function for a knot in Seifert manifolds}

\author{Hiroyuki Fuji}%
\address{Department of Information Systems,
Osaka Institute of Technology, 
Kitayama, Hirakata, Osaka, 573-0196, Japan}
\email{hiroyuki.fuji@oit.ac.jp}

\author{Kohei Iwaki}%
\address{Graduate School of Mathematical Sciences, 
The University of Tokyo, 
3-8-1 Komaba, Meguro-ku, Tokyo, 153-8914, Japan}
\email{iwaki@ms.u-tokyo.ac.jp}

\author{Hitoshi Murakami}%
\address{Graduate School of Information Sciences,
Tohoku University,
Aramaki-aza-Aoba 6-3-09, Aoba-ku,
Sendai, 980-8579, Japan}
\email{hitoshi@tohoku.ac.jp}

\author{Yuji Terashima}%
\address{Graduate School of Science, Tohoku University,
Aramaki-aza-Aoba 6-3, Aoba-ku, Sendai, 980-8578, Japan}
\email{yujiterashima@tohoku.ac.jp}

\date{}
\def\sl{{X (p_1/q_1,\dots,p_n/q_n)}}

\def\wf{{\Phi (q;N)}}

\def\qm{{\hat{\mathfrak m}}}
\def\ql{{\hat{\mathfrak l}}}

\begin{document}
\maketitle
%\pagestyle{empty}

% \today

\dedicatory{\it Dedicated to the memory of Toshie Takata}

\begin{abstract}
In this paper, for a Seifert loop (i.e., a knot in a Seifert three-manifold), 
first we give a family of an explicit function $\Phi(q; N)$ 
whose special values at roots of unity are identified with the Witten--Reshetikhin--Turaev 
invariants of the Seifert loop for the integral homology sphere. 
Second, we show that the function $\Phi(q; N)$ satisfies a $q$-difference equation
whose classical limit coincides with a component of the character varieties of the Seifert loop. 
Third, we give an interpretation of the function $\Phi(q; N)$ from the view point of the resurgent analysis.  
\end{abstract}

%%%%%%%%%%%%%%%%%%%%%%%%%%%%%%%%%%%%%%%%%%%%%%%%%%%%%%%%%%%%%%%%%%%%%%%%%%%%%
%%%%%%%%%%%%%%%%%%%%%%%%%%%%%%%%%%%%%%%%%%%%%%%%%%%%%%%%%%%%%%%%%%%%%%%%%%%%%
\section{Introduction}

% WRT invariant 

Given a knot in a three-manifold and a level $K \in {\mathbb N}$, 
Reshetikhin--Turaev constructed the quantum invariant through the representation 
of quantum groups \cite{RT91}. This invariant is closely related  
to the Chern-Simons path integral which was investigated by Witten \cite{Witten89}. 
The quantum invariant is now called the Witten--Reshetikhin--Turaev (WRT) invariant. 
Their discovery triggered an active interaction between topology and 
mathematical physics that continue to this day. 

% WRT function 

In this article, we will investigate several properties of 
a family of certain functions $\{ \Phi(q; N) \}_{N \in {\mathbb N}}$, 
which are defined as $q$-series on the unit disk $|q| < 1$, 
associated with a Seifert loop $\sl = (M, L)$ (i.e., a knot $L$ in a Seifert manifold $M$). 
Here the underlying Seifert manifold $M$ is obtained from $S^3$ through 
a (partial) rational surgery along the surgery diagram depicted in 
Figure \ref{fig:Seifert-loop} below.
Therefore, $p_1, q_1, \dots, p_n, q_n$ determine the topological type of 
the underlying Seifert manifold $M$, and $n$ denotes the number of singular fibers in $M$. 
We will impose a condition on $p_i$'s and $q_i$'s so that the underlying Seifert manifold 
is an integral homology sphere. 

We call the $q$-series $\Phi(q; N)$ the  ($SU(2)$) WRT function with the color $N$ 
since it is closely related to the ($SU(2)$) WRT invariant $\tau(K; N)$ 
for the Seifert loop with the color $N$ 
(which was studied by Lawrence--Rozansky \cite{LR} when $N=1$
and Beasley \cite{Beasley09} when $N \ge 1$). 
Namely, for any level $K \in {\mathbb N}$ and any color $N \in {\mathbb N}$, 
the radial limit of $\Phi(q; N)$ when $q$ tends to the root of unity $\exp(2\pi i /K)$ 
coincides with $\tau(K; N)$. More precisely, we have the following relation 
(see Theorem \ref{thm:q-series-and-quantum-invariant}): 
\begin{equation} \label{eq:radial-limit-theorem-in-introduction}
\lim_{t \to +0} \Phi(e^{\frac{2 \pi i }{K}} \, e^{-t}; N) 
= \tau(K; N).
\end{equation}
Hence, the WRT function can be regarded as an analogue of the 
$N$-colored Jones polynomials; indeed, when $n = 2$, 
the WRT function for $X(p_1/q_1, p_2/q_2)$ 
coincides with the $N$-colored Jones polynomial 
for the $(p_1, p_2)$-torus knot up to a certain normalization factor
(see Section \ref{subsection:WRT-function}).

The idea of describing the sequence of WRT invariants $\{ \tau(K; N) \}_{K \in {\mathbb N}}$ 
as limit values of a single $q$-series at roots of unity 
goes back to the work of Lawrence--Zagier \cite{LZ}, where the WRT invariant of 
the Poincar\'e homology sphere $\Sigma(2,3,5)$ was realized as the limit value 
of a $q$-series obtained as the Eichler integral of a modular form with a half-integral weight.
The $q$-series agrees with the WRT function in the case $N = 1$, $n = 3$ and $(p_1, p_2, p_3) = (2,3,5)$.
Generalizations of \cite{LZ} to Seifert manifolds were discussed by Hikami in his series of papers 
\cite{Hikami04, Hikami04-2, Hikami06, Hikami05, Himami11}. 

More recently, the idea was further developed by Gukov--Putrov--Vafa \cite{GPV16} based on 
very interesting perspectives from theoretical physics. 
They conjectured that there exists a decomposition of WRT invariant of a 3-manifold $M_3$ such that 
(a certain ``$S$-transform" of) the summands are labeled by $a \in {\rm Tor} H_1(M_3; {\mathbb Z})$ and 
given in terms of the limit value of a certain $q$-series, denoted by $\hat{Z}_{a}(q)$, 
which has integer coefficients. They also claim that there is a 3-manifold analogue 
of Khovanov-type homology which categorifies $\hat{Z}_a$. 
In \cite{GPPV17}, Gukov--Pei--Putrov--Vafa also discussed an analogue of $\hat{Z}_{a}(q)$ for 
a 3-manifold with a colored knot inside.  
Gukov--Manolescu also introduced a closely related two-variable series $F_{K}(x;q)$ in \cite{GM19}. 
For Seifert loops, we strongly believe that our WRT function $\Phi(q; N)$ is essentially 
the same as $\hat{Z}_a(q)$ with $a$ being the trivial connection, 
and our Theorem \ref{thm:q-series-and-quantum-invariant} provides a rigorous proof of 
the Conjecture 2.1 and Conjecture 4.1 
(in particular, the equality (A.27)) 
in \cite{GPPV17} for Seifert manifolds/loops 
with an arbitrary number of singular fibers.

We will also derive a $q$-difference equation satisfied by the WRT function. 
That is, we find a $q$-difference operator 
$\hat{A}(\hat{\mathfrak m}, \hat{\mathfrak l}; q)$
which annihilates the WRT function 
(see Theorem \ref{q-difference}): 
\begin{equation}
\hat{A}(\hat{\mathfrak m}, \hat{\mathfrak l}; q) \Phi(q; N) = 0.
\end{equation}
Here $\hat{\mathfrak m}$ and $\hat{\mathfrak l}$ acts as 
$\hat{\mathfrak m} \, \Phi(q;N) = q^{N/2} \, \Phi(q; N)$ 
and $\hat{\mathfrak l} \, \Phi(q;N) = \Phi(q; N+1)$, 
and they satisfy the $q$-commutation relation ${\ql \qm} = q^{1/2} {\qm \ql}$.
We will also confirm in Theorem \ref{AJ} that the classical limit 
of the $q$-difference operator is a component of the zero locus of 
the $A$-polynomial for the Seifert loop $\sl$. 
We note that the two-variable series $F_K(x,q)$ in \cite{GM19} is also expected to 
posses a similar property (see Conjecture 1.6 in \cite{GM19}). 
This observation is closely related to the AJ-conjecture \cite{Garoufalidis} 
which claims that the colored Jones polynomial for a knot in $S^3$ 
satisfies a $q$-difference equation whose classical limit coincides with 
the $A$-polynomial for the knot.
See also \cite{Gukov04} where a physical interpretation of the AJ-conjecture was given 
as the quantum volume conjecture.  
We also note that our WRT function is not a polynomial in general
unlike the colored Jones polynomial for a knot in $S^3$.

Finally, we give an alternative expression of the WRT function through the 
resurgent analysis for the perturbative part of the WRT invariant. 
Here, the perturbative part is a formal series obtained as the asymptotic expansion 
when $K \to + \infty$ of the part of WRT invariant which captures the contribution 
from the trivial connection on the Seifert loop. 
We borrow the ideas of Costin--Garoufalidis \cite{CG11}, 
Gukov--Marin\~o--Putrov \cite{GMP16}, 
Chun \cite{Chun17} and Chung \cite{Chung18}. 
These articles showed that a certain average of the Borel sums 
(median summation) of the perturbative part of the WRT invariant 
gives a $q$-series which has a nice modular property.
We will see that the median summation of the perturbative part of 
the WRT invariant of the Seifert loop $\sl$ coincides with 
the WRT function $\Phi(q;N)$ up to an overall factor through the change 
$q = \exp(2 \pi i/K)$ of the variables (see Theorem \ref{thm:WRT-function-as-Borel-sum}). 
Our computation relies on the integral expression of 
(the perturbative part of) the WRT invariant for the Seifert loops obtained by 
Lawrence--Rozansky \cite{LR} for $N = 1$ and Beasley \cite{Beasley09} for $N \ge 1$.
Our computation agrees with the comments in \cite{GMP16, GPPV17} which claim that the 
previously mentioned $q$-series $\hat{Z}_a$ is also computed via resurgent analysis.

We also note that the WRT invariant of Seifert manifolds for 
arbitrary finite dimensional complex simple Lie algebra was studied by 
Hansen--Takata \cite{Hansen-Takata-1, Hansen-Takata-2}
and Marin\~o \cite{Marino05}. 
It seems to be interesting to generalize the results in this paper 
and test the conjectures in \cite{GPPV17} 
for the WRT invariant of Seifert manifolds and Seifert loops 
for these Lie algebras.

This paper is organized as follows. 
In Section \ref{section:WRT-function-and-WRT-invariant}, we will introduce 
the WRT function for the Seifert loops and show that the WRT invariant 
is its radial limit. (We need a couple of technical lemmas which are proved in appendix). 
The $q$-difference equation satisfied by the WRT function will be derived 
in Section \ref{section:q-diffeence-WRT-invariant}. 
We will also discuss our partial proof of the AJ conjecture there. 
Section \ref{section:resurgence} will be devoted to the resurgent analysis 
of the perturbative part of the WRT invariant of the Seifert loop, where 
we will derive the WRT function through the Borel (median) summation. 

\begin{rem}
After the submission of this paper, we are informed that the paper \cite{AP} 
of Andersen--Misteg$\mathring{\rm a}$rd has been updated and contains overlapping results. 
They analyzed the Gukov--Pei--Putrov--Vafa's $q$-series $\hat{Z}_0(q)$ for Seifert homology spheres  
with arbitrary number of singular fibers, corresponding to the class $a = 0$. 

Firstly, (a normalization of) the WRT function $\Phi(q;N)$ with $N = 1$ has already appeared 
in the thesis \cite[Section 7.1.5]{Mistegard} of Misteg$\mathring{\rm a}$rd, and its coincidence with 
$\hat{Z}_0(q)$ for Seifert homology spheres was also expected there. 
This conjecture was proved by Andersen--Misteg$\mathring{\rm a}$rd in the new version \cite[p.\,25]{AP} 
through an explicit calculation of $\hat{Z}_0(q)$; 
thus we can now identify the WRT function $\Phi(q;N)$ with $\hat{Z}_0(q)$, at least when $N = 1$, thanks to their work.  
Andersen--Misteg$\mathring{\rm a}$rd also gave an expression of $\hat{Z}_0(q)$ through the resurgent analysis, 
which agrees with our Theorem \ref{thm:WRT-function-as-Borel-sum} in that case. 
Furthermore, under the assumption that one of $p_i$ is even, \cite{AP} also gave a proof of the radial limit property 
\eqref{eq:radial-limit-theorem-in-introduction} of $\hat{Z}_0(q)$ through a different method from the one used in this paper. 
These results were announced in the (online) talks \cite{Andersen-talk, Mistegard-talk} by the authors of \cite{AP}. 
\end{rem}

%%%%%%%%%%%%%%%%%%%%%%%%%%%%%%%%%%%%%%%%%%%%%%%%%%%%%%%%%%%%%%%%%%%%%%%%%%%%%
\subsection*{Acknowledgement}
The authors are grateful to William Elb$\ae$k Misteg$\mathring{\rm a}$rd, who kindly shear the new version of \cite{AP} 
and informed us that the WRT function $\Phi(q;N)$ and $\hat{Z}_0(q)$ are essentially the same $q$-series at least when $N=1$.  
We also thank 
Kazuhiro Hikami, 
Masaya Kameyama, 
Nobushige Kurokawa,
Serban Mihalache,
Akihito Mori,
Nobuo Sato 
and
Sakie Suzuki 
for valuable comments and discussions.
This work is partially supported by JSPS KAKENHI Grant Numbers 
JP16H03927, % HF（樋上さんの基盤Bの分担）
JP16H06337, % KI
JP17H06127, % KI
JP17K05239, % HM
JP17K05243, % YT
JP18K03281, % HF（佐竹さんの基盤Cの分担）
JP20K03601, % HM
JP20K03931, % HF（自分がメインの基盤C）
JP20K14323. % KI 

The authors would also express their deepest appreciation to Toshie Takata, 
who passed away on April 11th, 2020.
She was one of the pioneers of Quantum Topology.

%%%%%%%%%%%%%%%%%%%%%%%%%%%%%%%%%%%%%%%%%%%%%%%%%%%%%%%%%%%%%%%%%%%%%%%%%%%%%
%%%%%%%%%%%%%%%%%%%%%%%%%%%%%%%%%%%%%%%%%%%%%%%%%%%%%%%%%%%%%%%%%%%%%%%%%%%%%
\section{WRT invariant and WRT function for Seifert loops} 
\label{section:WRT-function-and-WRT-invariant}

In this section, we introduce an explicit $q$-series $\wf$, 
labeled by $N \ge 1$, 
whose special values at $K$-th root of unities 
are identified with the $SU(2)$ Witten--Reshetikhin--Turaev (WRT) invariants 
of the Seifert loop $\sl$ with level $K$ and color $N$.

%%%%%%%%%%%%%%%%%%%%%%%%%%%%%%%%%%%%%%%%%%%%%%%%%%%%%%%%%%%%%%%%%%%%%%%%%%%%%
\subsection{WRT invariant for Seifert loops}
\label{section:WRT-invariant}

Here we summarize several facts on the $SU(2)$ Witten--Reshetikhin--Turaev (WRT) invariant 
for the Seifert loop, which is a pair of the Seifert manifold $M$ and a knot $L$ inside of $M$. 
We denote by $\sl=(M,L)$ the Seifert loop, where the integers 
$p_1, \dots, p_n$ and $q_1, \dots, q_n$ specify the topological type 
of the Seifert manifold $M$ with $n$-singular fibers. 
More precisely, we take pairwise coprime integers $p_1, \dots, p_n \ge 2$
and integers $q_1, \dots, q_n$ satisfying  
\begin{equation} \label{eq:integral-homology-condition}
p_1 \cdots p_n \sum_{i=1}^{n} \frac{q_i}{p_i} = 1.
\end{equation}
Then, $\sl$ is obtained by a partial rational surgery along a link 
$L_0 \cup L_1 \cup \cdots L_n \cup L$ inside $S^{3}$ depicted in Figure \ref{fig:Seifert-loop}.
Here, the surgery indices of $L_0, L_1, \dots, L_n$ 
are $0, p_1/q_1, \dots, p_n/q_n$, respectively, 
and we do not apply the surgery along the last component $L$. 
The surgery along $L_0 \cup L_1 \cup \cdots L_n$ gives the 
Seifert manifold $M$, while $L$ remains as a knot in $M$. 
The assumption \eqref{eq:integral-homology-condition} guarantees 
that our Seifert manifold $M$ is an integral homology sphere. 

\begin{figure}[t]
\begin{center} 
  \includegraphics[width=7.0cm]{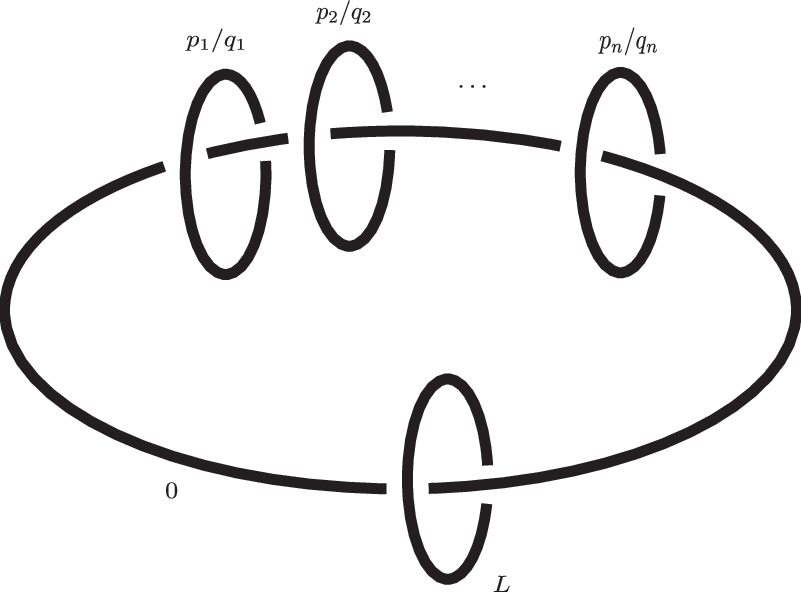} 
  \caption{Surgery diagram for the Seifert loop $\sl$.}
  \label{fig:Seifert-loop}
\end{center}  
\end{figure}

The WRT invariant of a general pair $(M,L)$ of any $3$-manifold $M$ 
and any framed colored knot $L$ in $M$ was defined in \cite[page 560]{RT91}. 
In \cite{LR}, Lawrence--Rozansky explicitly computed WRT invariant for 
the Seifert manifold $M$ when the knot $L$ is absent. 
A similar consideration at section 4 in \cite{LR} to the
$0$-framed Seifert loop $\sl$ 
gives the following explicit expression of the WRT invariant 
with any level $K \in {\mathbb Z}_{\ge 1}$ and 
with any color $N \in {\mathbb Z}_{\ge 1}$ along $L$:
\begin{align}
& \tau_{\sl}(K; N) = 
\frac{B G_0(K)}{K} \,
e^{- \frac{\pi i}{2K} \bigl( \Theta_0 + (N^2 - 1) P \bigr)}
\notag \\[-.0em] 
& \qquad \times \sum_{\substack{k = 0 \\ K |\hspace{-.3em}/ k}}^{2PK-1}
e^{- \frac{\pi i}{2K P}k^2}  \, 
\frac{
e^{N \frac{\pi i k}{K}}  - e^{-N \frac{\pi i k}{K}}}
{e^{\frac{\pi i k}{K}}  - e^{-\frac{\pi i k}{K}} }  
\, 
\frac{\displaystyle \prod_{j=1}^{n} 
\left( e^{\frac{\pi i k}{K p_j}}  - e^{-\frac{\pi i k}{K p_j}} \right)}
{\displaystyle \left( e^{\frac{\pi i k}{K}}  - e^{-\frac{\pi i k}{K}} \right)^{n-2}}.
\label{eq:SU2-WRT-invariant}
\end{align}
Here, we have used the same notations in \cite{LR}:
\begin{align}
P & := p_1 \cdots p_n, \\
B & := - \frac{1}{4\sqrt{P}} \, e^{\frac{3 \pi i}{4}}, \\
\Theta_0 & := 3 - \frac{1}{P} 
+ 12 \sum_{j=1}^{n} s(q_j, p_j), \\
G_0(K) & := \sqrt{\frac{K}{2}} \,
\frac{1}{\sin(\pi / K)}.
\label{eq:G-0}
\end{align}
where $s(q,p)$ is the Dedekind sum defined by 
\begin{equation}
s(q,p) := \frac{1}{4p} \sum_{\ell=1}^{p-1} \cot\Bigl( \frac{\pi \ell}{p} \Bigr) \, 
\cot \Bigl( \frac{\pi \ell q}{p} \Bigr).
\end{equation}
Note that the WRT invariant for the Seifert manifold $M$, which was computed in 
\cite[eq.\,(4.2)]{LR}, corresponds to the $N = 1$ case. 
(Our \eqref{eq:SU2-WRT-invariant} with $N=1$ is denoted by $Z_{K}(M)$ in \cite{LR}. 
We will use $Z$ for differently normalized WRT invariant below.)
In what follows, we will use simpler notation $\tau(K; N)$ for \eqref{eq:SU2-WRT-invariant}.

We will also use the WRT invariant with an alternative normalization:
\begin{equation}
\label{eq:WRT-invariant-Seifert-loop}
Z(K; N) := \frac{\tau(K;N)}{G_0(K)}.
\end{equation}
We note that $G_0(K)$ coincides with the WRT invariant 
$\tau_{S^1 \times S^2}(K)$ for $S^1 \times S^2$. 
We note that we have used the WRT invariant with 
the normalization $\tau_{S^3}(K) = 1$ (c.f., \cite{KM91}).

\begin{rem}
Applying the similar technique used in the proof of \cite[Theorem 1 or eq.\,(4.8)]{LR}, 
we can derive the following integral expression (including a residual part) 
of the WRT invariant for the Seifert loop $\sl$: 
\begin{align}
&  Z(K; N) = \frac{B}{2 \pi i} \, 
e^{- \frac{\pi i}{2K} \bigl( \Theta_0 + (N^2 - 1) P \bigr)}
\notag \\ 
& \hspace{+.5em} \times \Biggl[ \int_{{\mathbb R}  e^{\frac{\pi i}{4}}} 
e^{K  g(y)} 
F_N(y) \, dy 
 - 2 \pi i  \sum_{m=1}^{2P-1} \Res_{y= 2m\pi i} 
\frac{e^{K g(y)}  F_N(y)}
{1 - e^{-K y}} \, dy
\Biggr].
\label{eq:BLR-integral}
\end{align}
Here, 
\begin{align}
g(y) & := \frac{i}{8 \pi P} y^2,
\label{eq:def-of-g}
\\[+.5em]
F_N(y) & := 
\dfrac{\displaystyle 
\left( e^{\frac{N y}{2}} - e^{-\frac{N y}{2}} \right)
\prod_{j=1}^{n}
\left( e^{\frac{y}{2p_j}} - e^{-\frac{y}{2p_j}} \right)}
{\left( e^{\frac{y}{2}} - e^{-\frac{y}{2}} \right)^{n-1}}.
\label{eq:def-of-FN}
\end{align}
Note that the integral expression was also derived by Beasley 
in \cite[eq.\,(7.62)]{Beasley09} 
through the localization of Chern--Simons path integral. 
As is mentioned in \cite[page 302]{LR} and \cite[eq.\,(7.64)]{Beasley09}, 
the integral 
\begin{equation} \label{eq:trivial-connection-part}
Z_{\rm triv}(K; N) := 
\frac{B}{2 \pi i} \, e^{- \frac{\pi i}{2K} \bigl( \Theta_0 + (N^2 - 1) P \bigr)}
\int_{{\mathbb R}  e^{\frac{\pi i}{4}}} 
e^{K  g(y)} 
F_N(y) \, dy 
\end{equation}
is the contribution from the trivial connections on $\sl$ to the WRT invariant. 
In Section \ref{section:resurgence}, following the idea of 
Costin--Garoufalidis \cite{CG11} and 
Gukov--Marin\~o--Putrov \cite{GMP16}, 
we will use the asymptotic expansion of $Z_{\rm triv}(K; N)$ when $K \to +\infty$ 
to recover full information (including other flat connections) of the WRT invariant. 
We also note that, when $n = 2$, 
the formula \eqref{eq:BLR-integral} coincides with 
\eqref{eq:trivial-connection-part} 
(c.f., \cite[Section 7 and Appendix B]{Beasley09}), 
and the integral expression was effectively used to test 
the volume conjecture (\cite{Kashaev, Murakami-Murakami01}) 
for torus knots in $S^3$; 
see 
\cite{KT99, Murakami04, Hikami-Murakami08, Hikami-Murakami10}.
See also \cite{MMOTY, Gukov04, Murakami06} and the monograph \cite{MY}
for the complexified version and the generalized version
of the volume conjecture. 
 
\end{rem}

%%%%%%%%%%%%%%%%%%%%%%%%%%%%%%%%%%%%%%%%%%%%%%%%%%%%%%%%%%%%%%%%%%%%%%%%%%%%%
\subsection{WRT function for Seifert loops and values at roots of unity}
\label{subsection:WRT-function}

\begin{defi} 
For a positive integer $N$, we define an $N$-colored WRT function $\wf$ 
of a Seifert loop $\sl$ as the following $q$-series:
\begin{align}
& \wf :=\frac{(-1)^n}{2(q^{\frac{1}{2}} - q^{-\frac{1}{2}})} 
\,\, q^{-\frac{1}{4}(\Theta_0+(N^2-1)P)}
 \notag \\
& ~~ \times \sum_{\ell = - \frac{N-1}{2}}^{\frac{N-1}{2}} 
\sum_{(\varepsilon_1, \dots, \varepsilon_n) \in \{\pm 1 \}^n} 
\varepsilon_1 \cdots \varepsilon_n \,
\sum_{m=0}^{\infty} \dbinom{m+n-3}{n-3} \, 
q^{\frac{P}{4} (2m + 2 \ell + n-2 + 
\sum_{j=1}^{n}\frac{\varepsilon_j}{p_j})^2}. 
\label{eq:def-WRT-function}
\end{align}
\end{defi}

Let us give remarks on the WRT function.
\begin{rem} 
\begin{itemize}
\item[(i)] 
When $n = 1$ and $2$, we understand that the binomial coefficient 
in \eqref{eq:def-WRT-function} to be 
\begin{equation}
\dbinom{m+n-3}{n-3} = 
\begin{cases}
1 & \text{if $m = 0$} \\
0 & \text{if $m \ge 1$}.
\end{cases}
\end{equation}
Therefore, for $n = 1$ and $2$, 
the right-hand side of \eqref{eq:def-WRT-function}
becomes a finite sum. 
In particular, when $n=2$, we have
\begin{equation}
\Phi_{n=2}(q;N) = 
\frac{q^{\frac{N}{2}} - q^{- \frac{N}{2}}}{q^{\frac{1}{2}} - q^{- \frac{1}{2}}} \, 
J_{T_{p_1, p_2}}(q;N)
\end{equation}
with 
\begin{align}
& J_{T_{p_1, p_2}}(q;N) := 
\frac{q^{\frac{p_1 p_2}{4}(1-N^2)}}{q^{\frac{N}{2}} - q^{- \frac{N}{2}}} 
\notag \\ 
& \qquad \times 
\sum_{\ell = - \frac{N-1}{2}}^{\frac{N-1}{2}} 
\left(
q^{p_1p_2 \ell^2-(p_1+p_2)\ell+\frac{1}{2}} -  
q^{p_1p_2 \ell^2-(p_1-p_2)\ell-\frac{1}{2}}
\right).
\end{align}
Note that $J_{T_{p_1, p_2}}(q;N)$ is the colored Jones polynomial 
for the $(p_1, p_2)$-torus knot. 
(See \cite{RJ93}, \cite[Section 3]{Morton95} and 
\cite{Hikami-Kirillov} for example.)
The above colored Jones polynomial is normalized as 
$J_{\rm unknot}(q;N) = 1$. 
From these facts, we may regard our $\Phi(q;N)$ 
as a generalization of colored Jones polynomial which is 
normalized so that it gives the $q$-integer 
$(q^{\frac{N}{2}} - q^{- \frac{N}{2}})/(q^{\frac{1}{2}} - q^{- \frac{1}{2}})$
for the unknot.

\item[(ii)]
When $n = 3$ an $N=1$, the $q$-series 
$2 (q^{\frac{1}{2}} - q^{-\frac{1}{2}}) q^{\frac{1}{4}(\Theta_0+(N^2-1)P)} \, \Phi(q; N)$ 
is specialized to be the Eichler integral of a modular form with a half-integral weight 
which was considered in Lawrence--Zagier \cite{LZ} and Hikami \cite{Hikami04}.

\item[(iii)] 
As we mentioned in the introduction, we expect that the WRT function 
is essentially the same as the $q$-series $\hat{Z}_a$ in \cite[Section 4]{GPPV17}. 
This is true when $N = 1$ due to the recent work \cite{AP} of Andersen--Misteg$\mathring{\rm a}$rd.
We also expect that our $\Phi(q;N)$ agrees with the two-variable series $F_K(x,q)$ studied in \cite{GM19} with $x = q^{N}$ 
in the presence of a knot $L$ inside the Seifert manifold (c.f., \cite{Chung20}). 
\end{itemize}
\end{rem}

% \bigskip
The main result in this section is

\begin{thm}\label{thm:q-series-and-quantum-invariant}
For each $K \in {\mathbb Z}_{\ge 1}$, we have
\begin{equation} \label{eq:q-series-limiting-value}
\lim_{t \rightarrow 0+} \Phi\left(e^{\frac{2\pi i}{K}} \, e^{-t}; N \right) 
= 
% e^{- \frac{\pi i}{2K} (N^2-1) P} \, 
% (e^{\frac{\pi i}{K}} - e^{- \frac{\pi i}{K}}) \, \tau(K;N).
\tau(K;N).
\end{equation}
\end{thm}

We note that some special cases of Theorem \ref{thm:q-series-and-quantum-invariant} 
were proved in previous works.
Lawrence--Zagier \cite{LZ} proved the statement for the Poincar\'e homology sphere 
(i.e.,  $N = 1$, $n=3$ and $(p_1, p_2, p_3) = (2,3,5)$), 
and Hikami \cite{Hikami04} also gave a proof for the Brieskorn homology spheres  
(i.e.,  $N = 1$, $n=3$ and general pairwise coprime triple $(p_1, p_2, p_3)$). 
Theorem \ref{thm:q-series-and-quantum-invariant} suggests that 
the $q$-series $\Phi(q; N)$ is an ``analytic continuation" of 
the quantum invariant $\tau(K; N)$ with respect to $K$ from 
integers to complex numbers.
We will prove Theorem \ref{thm:q-series-and-quantum-invariant}
in the next subsection.

%%%%%%%%%%%%%%%%%%%%%%%%%%%%%%%%%%%%%%%%%%%%%%%%%%%%%%%%%%%%%%%%%%%%%%%%%%%%%
\subsection{Proof of Theorem \ref{thm:q-series-and-quantum-invariant}} 

For $k = 1, 2$ and $t \in {\mathbb C}$ with ${\rm Re}\, t >0$, define 
\begin{align}
\phi^{(k)}(t) & := (-1)^n
\sum_{\ell = - \frac{N-1}{2}}^{\frac{N-1}{2}} \,
\sum_{(\varepsilon_1, \dots, \varepsilon_n) \in \{\pm 1 \}^n} 
\varepsilon_1 \cdots \varepsilon_n 
 \notag \\[+.3em]
& \quad \times
\sum_{m=0}^{\infty} \dbinom{m+n-3}{n-3} \,
e^{\frac{\pi i}{2KP}(2Pm + a_{\ell, \varepsilon})^2} \,
e^{- (2Pm + a_{\ell, \varepsilon})^k \, t},
\end{align}
where $a_{\ell, \varepsilon} \in {\mathbb Z}$ is given by
\begin{equation} \label{eq:a-ell-epsilon}
a_{\ell, \varepsilon} := P \Bigl( 2 \ell + n-2 
+ \sum_{j=1}^{n}\frac{\varepsilon_j}{p_j}\Bigr).
\end{equation}
Note that 
\begin{equation}
\left[ 2(q^{\frac{1}{2}} - q^{- \frac{1}{2}}) 
q^{\frac{1}{4}(\Theta_0+(N^2-1)P)} \, \Phi(q; N) 
\right]_{q = e^{\frac{2\pi i}{K}} \, e^{-t}} = 
\phi^{(2)}\Bigl(\frac{t}{4P}\Bigr).
\end{equation}

Following the idea of \cite{LZ} and \cite{Hikami04},  
we will prove Theorem \ref{thm:q-series-and-quantum-invariant} 
along the following scheme. 

% \newpage

\begin{prop} \label{prop:key-prop} 
\begin{itemize}
\item[{\rm (i)}]
The function $\phi^{(1)}(t)$ has an asymptotic expansion 
of the form
\begin{equation} \label{eq:asymptotic-phi}
\phi^{(1)}(t) \sim \sum_{r=0}^{\infty} b^{(1)}_{r} t^{r}
\end{equation}
when $t \rightarrow 0+$. 
Furthermore, the limit value $b_0^{(1)} = \lim_{t \to 0+} \phi^{(1)}(t)$ 
is proportional to the right-hand side of 
\eqref{eq:q-series-limiting-value}:
\begin{equation} \label{eq:b0-1}
b_0^{(1)} = \frac{e^{\frac{\pi i}{4}}}{\sqrt{2KP}}  \,
\sum_{\substack{k = 0 \\ K |\hspace{-.3em}/ k}}^{2PK-1}
e^{- \frac{\pi i}{2K P}k^2} \, 
\frac{
e^{N \frac{\pi i k}{K}}  - e^{-N \frac{\pi i k}{K}}}
{e^{\frac{\pi i k}{K}}  - e^{-\frac{\pi i k}{K}} }  
 \, 
\frac{\displaystyle \prod_{j=1}^{n} 
\left( e^{\frac{\pi i k}{K p_j}}  - e^{-\frac{\pi i k}{K p_j}} \right)}
{\displaystyle \left( e^{\frac{\pi i k}{K}}  - e^{-\frac{\pi i k}{K}} \right)^{n-2}}.
\end{equation}

\item[{\rm (ii)}] 
The function $\phi^{(2)}(t)$ also has an asymptotic expansion 
of the form
\begin{equation} \label{eq:asymptotic-phi-2}
\phi^{(2)}(t) \sim \sum_{r=0}^{\infty} b^{(2)}_{r} t^{r}
\end{equation}
when $t \rightarrow 0+$. Moreover, 
the limit value $b_{0}^{(2)} = \lim_{t \rightarrow 0+} \phi^{(2)}(t)$ 
coincides with that of $\phi^{(1)}(t)$:
\begin{equation} \label{eq:coincidence-of-leading}
b_{0}^{(1)} = b_{0}^{(2)}.
\end{equation}
\end{itemize}
\end{prop}

Consequently, Theorem \ref{thm:q-series-and-quantum-invariant} 
follows from the equalities \eqref{eq:b0-1} and
\eqref{eq:coincidence-of-leading}.
We will prove these statements in the rest of this section.

%%%%%%%%%%%%%%%%%%%%%%%%%%%%%%%%%%%%%%%%%%%%%%%%%%%%%%%%%%%%%%%%%%%%%%%%%%%%%
\subsubsection{Proof of Proposition \ref{prop:key-prop} (i)} 
\label{subsection:proof-of-prop-1}

To derive the formula \eqref{eq:b0-1}, we use the idea of 
Hikami \cite{Hikami04, Hikami04-2, Hikami06}.
In particular, we will use the following quadratic reciprocity formula. 

\begin{lem}[{e.g., \cite[Section 2]{Hikami04}}] ~
For any $M_1, M_2 \in {\mathbb Z}$ and $L \in {\mathbb Q}$ satisfying
$M_1 \cdot M_2 \in 2{\mathbb Z}$ and $M_1 \cdot L \in {\mathbb Z}$, we have 
\begin{equation} \label{eq:reciprocity-1}
\sum_{k\,\,{\rm mod}\,\,M_1} e^{\pi i \, \frac{M_2}{M_1} \, k^2 + 2 \pi i \, L k} 
= 
\sqrt{\left| \frac{M_1}{M_2} \right|} \,
e^{\frac{\pi i}{4} \, {\rm sign}(M_1 \cdot M_2)} 
\, \sum_{k\,\,{\rm mod}\,\,M_2} e^{- \pi i \, \frac{M_1}{M_2} \, (k+L)^2}.
\end{equation}
\end{lem} 

If we set
\begin{equation}
G(KP) := \sum_{k\,\,{\rm mod}\,\, 2KP} e^{- \frac{\pi i}{2 KP}k^2},
\end{equation}
then we have
\begin{align}
G(KP) & = \sqrt{2KP} \, e^{- \frac{\pi i}{4}} \\ 
& = 
e^{- \frac{\pi i}{2KP} \tilde{L}^2} \,
\sum_{k\,\,{\rm mod}\,\, 2KP} 
e^{- \frac{\pi i}{2KP}  k^2 + \frac{\pi i}{KP} \tilde{L} k}. 
\label{eq:reciprocity-2}
\end{align}
We have applied the reciprocity formula for 
$M_1=2 KP, M_2 = -1, L = 0$ to obtain the first line, while 
$M_1=2 KP, M_2 = -1, L = \tilde{L}/(2KP)$ 
with an arbitrary integer $\tilde{L}$ to obtain the second line.

Keeping the formula in our mind,
let us compute the limit value $b_0^{(1)} = \lim_{t \to 0+} \phi^{(1)}(t)$. 
It follows from the definition of $\phi^{(1)}(t)$ that   
\begin{align}
& G(KP) \, \phi^{(1)}(t) \notag \\
%%%
& \quad = 
(-1)^n 
\sum_{\ell = - \frac{N-1}{2}}^{\frac{N-1}{2}} \,
\sum_{(\varepsilon_1, \dots, \varepsilon_n)} 
\varepsilon_1 \cdots \varepsilon_n 
 \notag \\
& \qquad \times ~~ \quad \sum_{m=0}^{\infty} \dbinom{m+n-3}{n-3} \, 
e^{\frac{\pi i}{2KP}(2Pm + a_{\ell, \varepsilon} )^2}  \, 
e^{- (2Pm + a_{\ell, \varepsilon}) \, t} \, 
\notag \\
& 
\qquad \times \quad 
e^{- \frac{\pi i}{2KP}(2Pm + a_{\ell, \varepsilon})^2} 
\sum_{k\,\,{\rm mod}\,\, 2KP} e^{- \frac{\pi i}{2KP}k^2
 + \frac{\pi i}{KP} (2Pm + a_{\ell, \varepsilon}) \, k}
\notag \\ 
& \quad = 
(-1)^n \sum_{k\,\,{\rm mod}\,\, 2KP} e^{- \frac{\pi i}{2KP}k^2} \, 
\sum_{\ell = - \frac{N-1}{2}}^{\frac{N-1}{2}} \,
\sum_{(\varepsilon_1, \dots, \varepsilon_n)} 
\varepsilon_1 \cdots \varepsilon_n \,
\frac{e^{\frac{\pi i k}{KP} a_{\ell,\varepsilon}} \, 
e^{- a_{\ell,\varepsilon} t} }
{\bigl( 1 - e^{\frac{2\pi i k}{K}} \, e^{-2 P t} \bigr)^{n-2}}
\end{align}
holds when ${\rm Re} \, t > 0$. 
We have used \eqref{eq:reciprocity-2} with $\tilde{L} = 2Pm + a_{\ell, \varepsilon}$
in the first equality.
 
To evaluate the limit value, we decompose the sum over $k$ into two parts:
\begin{align}
\phi_A^{(1)}(t) 
& := 
\frac{(-1)^n}{G(KP)} \, \sum_{\substack{k\,\,{\rm mod}\,\, 2KP \\ K |\hspace{-.3em}/ k}} 
e^{- \frac{\pi i}{2KP}k^2} \notag \\ 
& \quad \times \sum_{\ell = - \frac{N-1}{2}}^{\frac{N-1}{2}} \,
\sum_{(\varepsilon_1, \dots, \varepsilon_n)} 
\varepsilon_1 \cdots \varepsilon_n \, 
\frac{
e^{\frac{\pi i k}{KP} a_{\ell,\varepsilon}} \, 
e^{- a_{\ell,\varepsilon} t}
}
{\bigl( 1 - e^{\frac{2\pi i k}{K}} \, e^{-2 P t} \bigr)^{n-2}},
\label{eq:phi1-A} 
\\[+.5em]
\phi_B^{(1)}(t) 
& := 
\frac{(-1)^n }{G(KP)} \, \sum_{\substack{k\,\,{\rm mod}\,\, 2KP \\ K | k}} 
e^{- \frac{\pi i}{2KP}k^2} \notag \\
& \quad \times
\sum_{\ell = - \frac{N-1}{2}}^{\frac{N-1}{2}} \,
\sum_{(\varepsilon_1, \dots, \varepsilon_n)} 
\varepsilon_1 \cdots \varepsilon_n \,
\frac{
e^{\frac{\pi i k}{KP} a_{\ell,\varepsilon}} \,
e^{- a_{\ell,\varepsilon} t}
}
{\bigl( 1 - e^{\frac{2\pi i k}{K}} \, e^{-2 P t} \bigr)^{n-2}}
\notag \\
& = 
\frac{(-1)^n}{G(KP)} \, \frac{1}{\bigl( 1 - e^{-2 P t} \bigr)^{n-2}} 
\notag \\
& \quad \times 
\sum_{m\,\,{\rm mod}\,\,2P} 
e^{- \frac{\pi i K}{2P}m^2} \,
\sum_{\ell = - \frac{N-1}{2}}^{\frac{N-1}{2}} \,
\sum_{(\varepsilon_1, \dots, \varepsilon_n)} 
\varepsilon_1 \cdots \varepsilon_n \, 
e^{- a_{\ell,\varepsilon} t} \, 
e^{\frac{\pi i m}{P} a_{\ell, \varepsilon}}. 
\label{eq:phi1-B}
\end{align}  
Obviously, $\phi^{(1)}(t) = \phi_A^{(1)}(t) + \phi_B^{(1)}(t)$.

\bigskip
$\bullet$ \underline{Asymptotic behavior of $\phi_A^{(1)}(t)$.} \, 
It follows from the expression \eqref{eq:phi1-A} that 
$\phi_A^{(1)}(t)$ has the asymptotic expansion 
\begin{equation}
\phi_A^{(1)}(t) \sim \sum_{r = 0}^{\infty} b_{A,r}^{(1)} t^{r}
\end{equation}
as $t \to 0+$. A direct computation shows that the leading term is given by 
\begin{align}
b_{A,0}^{(1)} & = \lim_{t \to 0+} \phi_A^{(1)}(t) 
\notag \\
& = \frac{1}{G(KP)} \, 
\sum_{\substack{k\,\,{\rm mod}\,\, 2KP \\ K |\hspace{-.3em}/ k}} 
e^{- \frac{\pi i}{2KP}k^2} \notag \\
& \quad \times 
\sum_{\ell = - \frac{N-1}{2}}^{\frac{N-1}{2}} \,
\sum_{(\varepsilon_1, \dots, \varepsilon_n)} 
(-1)^n \varepsilon_1 \cdots \varepsilon_n \,
\frac{  e^{\frac{\pi i k}{K} \bigl(2 \ell + n-2 + 
\sum_{j=1}^{n}\frac{\varepsilon_j}{p_j} \bigr)} }
{\bigl( 1 - e^{\frac{2\pi i k}{K}} \bigr)^{n-2}} 
\notag \\
%%%
& = \frac{1}{G(KP)} \, \sum_{\substack{k\,\,{\rm mod}\,\, 2KP \\ K |\hspace{-.3em}/ k}} 
e^{- \frac{\pi i}{2KP}k^2} \,
\frac{
e^{N \frac{\pi i k}{K}}  - e^{-N \frac{\pi i k}{K}}}
{e^{\frac{\pi i k}{K}}  - e^{-\frac{\pi i k}{K}} }  \, 
\frac{\displaystyle \prod_{j=1}^{n} 
\left( e^{\frac{\pi i k}{K p_j}}  - e^{-\frac{\pi i k}{K p_j}} \right)}
{\displaystyle \left( e^{\frac{\pi i k}{K}}  - e^{-\frac{\pi i k}{K}} \right)^{n-2}}.
\label{eq:bA0-1}
\end{align}
Therefore, we have shown that the limit value $\lim_{t \to 0+} \phi_A^{(1)}(t)$ 
agrees with the right hand-side of \eqref{eq:b0-1}.

\bigskip
$\bullet$ \underline{Asymptotic behavior of $\phi_B^{(1)}(t)$.} \, 
Since $1/(1 - e^{-2 P t})^{n-2}$ has a pole of order $n-2$ at the origin, 
the expression \eqref{eq:phi1-B} implies that 
$\phi_B^{(1)}(t)$ has the asymptotic expansion of the following form:
\begin{equation} \label{eq:asymptotic-phi1-B}
\phi_B^{(1)}(t) \sim \sum_{r = -(n-2)}^{\infty} b_{B,r}^{(1)} t^{r}.
\end{equation}
Our task is to prove that the coefficients $b_{B,r}^{(1)}$
of non-positive powers of $t$ vanish. 

Using the quadratic reciprocity again (for $M_1 = 2P, M_2 = - K, 
L= \frac{a_{\ell,\varepsilon}}{2P} = \ell + \frac{n-2}{2} + 
\sum_{j=1}^{n}\frac{\varepsilon_j}{2p_j}$), 
we can modify the expression \eqref{eq:phi1-B} as follows:
\begin{align}
\phi_B^{(1)}(t) 
%%%
& = 
\frac{(-1)^n}{K} \, \frac{1}{\bigl( 1 - e^{-2 P t} \bigr)^{n-2}} \notag \\
& \quad \times 
\sum_{\ell = - \frac{N-1}{2}}^{\frac{N-1}{2}} \,
\sum_{(\varepsilon_1, \dots, \varepsilon_n)} 
\varepsilon_1 \cdots \varepsilon_n \, 
e^{- a_{\ell,\varepsilon} t} \, 
\sum_{m\,\,{\rm mod}\,\,K} 
e^{\frac{\pi i}{2KP}(2mP + a_{\ell,\varepsilon})^2}.
\end{align}
Therefore, the coefficient $b^{(1)}_{B, r}$ in \eqref{eq:asymptotic-phi1-B} 
is written by a linear combination of elements in 
\begin{equation}
\left\{  \sum_{\ell = - \frac{N-1}{2}}^{\frac{N-1}{2}} \,
\sum_{(\varepsilon_1, \dots, \varepsilon_n)} 
\varepsilon_1 \cdots \varepsilon_n \, a_{\ell, \varepsilon}^s \, 
\sum_{m\,\,{\rm mod}\,\,K} e^{\frac{\pi i}{2KP}\, ( 2Pm + a_{\ell, \varepsilon} )^2} ~;~
s \in \{0,1,\dots, r + n - 2  \}   \right\}.
\end{equation}

\begin{lem} \label{lemm:vanishing-sums}
For any $s \in \{0,1,\dots, n-2 \}$ and 
any $\ell \in \frac{1}{2} {\mathbb Z}$, 
we have
\begin{equation}
\sum_{(\varepsilon_1, \dots, \varepsilon_n) \in \{ \pm 1 \}^n} 
\varepsilon_1 \cdots \varepsilon_n \, 
\biggl( \sum_{j=1}^{n} \frac{\varepsilon_j}{p_j} \biggr)^s \, 
\sum_{m\,\,{\rm mod}\,\,K} e^{\frac{\pi i}{2KP}\, ( 2Pm + a_{\ell, \varepsilon} )^2}
= 0.
\end{equation}
\end{lem} 
We will give a proof of Lemma \ref{lemm:vanishing-sums} 
in Appendix \ref{appendix:vanishing-sums}. 

Lemma \ref{lemm:vanishing-sums} implies that 
the coefficients of non-positive powers of $t$ in 
\eqref{eq:asymptotic-phi1-B} vanish:
\begin{equation}
b_{B,-(n-2)}^{(1)} = \cdots = 
b_{B,-1}^{(1)} = b_{B,0}^{(1)} = 0.
\end{equation}
This guarantees the existence of the asymptotic expansion 
\eqref{eq:asymptotic-phi} of $\phi^{(1)}(t)$. 
Moreover, since $b_{0}^{(1)} = b_{A,0}^{(1)} + b_{B,0}^{(1)} = b_{A,0}^{(1)}$,
the desired equality \eqref{eq:b0-1} follows from \eqref{eq:bA0-1}.
This completes the proof of Proposition \ref{prop:key-prop} (i).

%%%%%%%%%%%%%%%%%%%%%%%%%%%%%%%%%%%%%%%%%%%%%%%%%%%%%%%%%%%%%%%%%%%%%%%%%%%%%
\subsubsection{Proof of Proposition \ref{prop:key-prop} (ii)} 
\label{subsection:proof-of-prop-2}

We employ the idea of \cite[Section 3]{LZ}.
To derive the asymptotic property \eqref{eq:asymptotic-phi-2} of $\phi^{(2)}(t)$, 
we consider the Mellin transforms of $\phi^{(1)}(t)$ and $\phi^{(2)}(t)$.

For $k=1,2$, we introduce
\begin{align}
L^{(k)}(s) & := (-1)^n  
\sum_{\ell = - \frac{N-1}{2}}^{\frac{N-1}{2}} \,
\sum_{(\varepsilon_1, \dots, \varepsilon_n)} 
\varepsilon_1 \cdots \varepsilon_n \,
\notag \\ & \quad  \times
\sum_{m=0}^{\infty} \dbinom{m+n-3}{n-3} \,
e^{\frac{\pi i}{2KP}(2Pm + a_{\ell,\varepsilon})^2}  \,
\bigl( 2Pm + a_{\ell,\varepsilon} \bigr)^{-ks} 
\label{eq:Dirichlet-L-series}
\end{align}
The asymptotic property \eqref{eq:asymptotic-phi} of $\phi^{(1)}(t)$
enables us to show that $L^{(1)}(s)$ defines an entire function of $s$, 
and its special value at $s=0$ is
\begin{equation} \label{eq:L1-at-origin}
L^{(1)}(0) = b_0^{(1)}.
\end{equation}
(See \cite[Section 7]{Zagier81}.)
On the other hand, since $L^{(2)}(s) = L^{(1)}(2s)$ by their definitions, we have
\begin{equation}
\phi^{(2)}(t) = \frac{1}{2 \pi i} 
\int^{+\delta+i \infty}_{+\delta - i \infty} 
\Gamma(s) \, L^{(1)}(2s) \, t^{-s} \,ds.
\end{equation}
Here $\delta$ is any positive real number, and the 
integration is taken along a line which is parallel to the imaginary axis. 
By moving this contour to the left across the simple poles at 
$s=0, -1, -2, \dots$, we obtain the asymptotic expansion 
\begin{equation}
\phi^{(2)}(t) \sim \sum_{r = 0}^{\infty} 
\left( \Res_{s=-r} \Gamma(s) \, L^{(1)}(2s) \, t^{-s} \,ds \right) \, t^{r}
= \sum_{r = 0}^{\infty} \frac{(-1)^r}{r!} L^{(1)}(-2r) \,  t^{r}
\end{equation}
when $t \rightarrow +0$. 
Thus we obtain \eqref{eq:asymptotic-phi} of $\phi^{(2)}(t)$. 
Moreover, the last formula implies that 
\begin{equation}
b_0^{(2)} = \lim_{t \rightarrow 0+} \phi^{(2)}(t) = L^{(1)}(0) = b_{0}^{(1)}.
\end{equation}
Here we have used \eqref{eq:L1-at-origin}.
This completes the proof of (ii) in Proposition \ref{prop:key-prop}, 
and hence, Theorem \ref{thm:q-series-and-quantum-invariant} is proved.

%%%%%%%%%%%%%%%%%%%%%%%%%%%%%%%%%%%%%%%%%%%%%%%%%%%%%%%%%%%%%%%%%%%%%%%%%%%%%
%%%%%%%%%%%%%%%%%%%%%%%%%%%%%%%%%%%%%%%%%%%%%%%%%%%%%%%%%%%%%%%%%%%%%%%%%%%%%
\section{$q$-difference equation for the WRT function and its classical limit}
\label{section:q-diffeence-WRT-invariant}

In this section, we obtain an explicit $q$-difference equation satisfied with 
the WRT function $\wf$ for the Seifert loop $X(p_1/q_1,\ldots,p_n/q_n)=(M,L)$. 
Moreover, we show that the classical limit of the $q$-difference equation 
is a component of the algebraic curve defined as the zero locus of 
the $A$-polynomial of the Seifert loop.
(See \cite{CCGLS94} for the definition of $A$-polynomial.)

%%%%%%%%%%%%%%%%%%%%%%%%%%%%%%%%%%%%%%%%%%%%%%%%%%%%%%%%%%%%%%%%%%%%%%%%%%%%%
\subsection{$q$-difference equation satisfied by the WRT function}
For a general family $F=\{ F(q;N) \}_{N}$ of $q$-series parametrized by 
positive integers $N$, 
we define $q$-difference operators $\qm, \ql$ by  
\begin{align}
(\qm F)(q;N)& :=q^{N/2} F(q;N), \\
(\ql F)(q;N)& :=F(q;N+1).
\end{align}
These operators satisfy the $q$-commutation relation:
\begin{equation}
{\ql \qm} = q^{\frac{1}{2}} {\qm \ql}. 
\end{equation}

\begin{thm}\label{q-difference}
The family $\Phi=\{ \wf \}_{N}$ of the WRT functions parameterized by 
the color $N$ satisfies the following $q$-difference equation:
\begin{equation} \label{eq:q-difference-equation-WRT}
\left[ \ql^3-q^{- \frac{P}{2}}\frac{C(q \, \qm)}{C(q^{\frac{1}{2}} \qm)}\ql^2-q^{-2P}
\qm^{-2P} \ql 
+q^{-\frac{3P}{2}}\frac{C(q \, \qm)}{C(q^{\frac{1}{2}} \qm)}\qm^{-2P}  \right] \Phi=0.
\end{equation}
Here, we set 
\begin{align}
C({\mathfrak m}):= % & 
\sum_{(\varepsilon_1, \dots, \varepsilon_n) \in \{\pm 1 \}^n} 
\varepsilon_1 \cdots \varepsilon_n \,
\sum_{m=0}^{\infty} \dbinom{m+n-3}{n-3} \, 
q^{\frac{a_{m,\varepsilon}^2}{4P}} \, 
\Bigl( 
{\mathfrak m}^{a_{m,\varepsilon}} + {\mathfrak m}^{-a_{m,\varepsilon}}
\Bigr).
\end{align}
(We remind the readers that 
$a_{m,\varepsilon} = P(2m+n-2+\sum_{j=1}^{n}(\varepsilon_j/p_j)) \in {\mathbb Z}$ 
was given in \eqref{eq:a-ell-epsilon}.)
\end{thm}

\begin{proof}
First, for an easier description, we write $\Phi(N) := \Phi(q;N)$ 
and introduce 
\begin{align}
D(N):&= (-1)^n \,
\frac{q^{-\frac{1}{4}(\Theta_0+(N^2-1)P)}}{2(q^{\frac{1}{2}} - q^{- \frac{1}{2}})}, 
 \\
R(\ell ):&=\sum_{(\varepsilon_1, \dots, \varepsilon_n) \in \{\pm 1 \}^n} 
\varepsilon_1 \cdots \varepsilon_n 
\sum_{m=0}^{\infty} \dbinom{m+n-3}{n-3} 
q^{\frac{1}{4P} (2 \ell P + a_{m, \varepsilon})^2}.
\end{align}
It follows from the definition of $\Phi(N)$ that
\begin{equation}
\Phi(N) = D(N) \, \sum_{\ell = - \frac{N-1}{2}}^{\frac{N-1}{2}} R(\ell)
\end{equation}
holds. Therefore, we have
\begin{align}
\frac{\Phi(N+2)}{D(N+2)} 
=\sum_{\ell = - \frac{N+1}{2}}^{\frac{N+1}{2}} R(\ell ) 
=\frac{\Phi(N)}{D(N)}+R\Bigl(\frac{N+1}{2} \Bigr) 
+ R\Bigl( - \frac{N+1}{2} \Bigr).
\end{align}
Introducing
\begin{equation}
\widetilde{C}(N):=D(N+2) \left( R\Bigl( \frac{N+1}{2} \Bigr) 
+ R\Bigl( - \frac{N+1}{2} \Bigr) \right), 
\end{equation}
we get a second order inhomogeneous $q$-difference equation 
\begin{equation} \label{eq:WRT-function-first-order-raletion}
\Phi(N+2)=q^{-P(N+1)} \Phi (N)+\widetilde{C}(N).
\end{equation}
Here we have used $D(N+2)/D(N) = q^{-P(N+1)}$.
To obtain a homogeneous $q$-difference equation satisfied by $\Phi$, 
we subtract both sides of
\begin{equation} \label{eq:C-Phi}
\widetilde{C}(N)\Phi(N+3)=q^{-P(N+2)} 
\widetilde{C}(N)\Phi (N+1)+\widetilde{C}(N)\widetilde{C}(N+1)
\end{equation}
from both sides of
\begin{equation}
\widetilde{C}(N+1)\Phi(N+2)=q^{-P(N+1)} 
\widetilde{C}(N+1)\Phi (N)+\widetilde{C}(N+1)\widetilde{C}(N).
\end{equation}
Then, we get 
\begin{align}
\label{eq:pre-q-difference-equation}
& \widetilde{C}(N+1)\Phi(N+2)-\widetilde{C}(N)\Phi(N+3) \notag \\
&\qquad =q^{-P(N+1)} \widetilde{C}(N+1) \Phi (N)
-q^{-P(N+2)} \widetilde{C}(N)\Phi (N+1).
\end{align}
Since 
\begin{align}
& R\Bigl( \frac{N+1}{2} \Bigr) 
+ R\Bigl( - \frac{N+1}{2} \Bigr) 
\notag \\ 
& \quad 
= q^{\frac{P}{4} (N+1)^2}
\sum_{(\varepsilon_1, \dots, \varepsilon_n)} 
\varepsilon_1 \cdots \varepsilon_n 
\sum_{m=0}^{\infty} \dbinom{m+n-3}{n-3} \, 
q^{\frac{a_{m,\varepsilon}^2}{4P}} \, 
\Bigl( 
q^{\frac{a_{m,\varepsilon}(N+1)}{2}} 
+ q^{- \frac{a_{m,\varepsilon}(N+1)}{2}}
\Bigr),
\end{align}
we have
\begin{align}
\frac{\widetilde{C}(N+1)}{\widetilde{C}(N)} 
= q^{- \frac{P}{2}} \left[\frac{C(q \,{\mathfrak m})}{C(q^{\frac{1}{2}} {\mathfrak m})} 
\right]_{{\mathfrak m}=q^{N/2}}.
\end{align}
Thus, we may express the $q$-difference equation 
\eqref{eq:pre-q-difference-equation} by using the operators $\qm$ and $\ql$. 
Consequently, we have the desired equality \eqref{eq:q-difference-equation-WRT}.
This completes the proof of Theorem \ref{q-difference}. 
\end{proof}

\begin{rem}
As is shown in \cite[Proposition 5, Theorem 6]{Hikami-difference}, 
a simplification happens to the $q$-difference equation 
when $n=2$ and one of $p_1$ or $p_2$ equals to $2$.
Namely, the WRT function for $n=2$ and $(p_1, p_2) = (2, 2k+1)$
(or the colored Jones polynomial for the $(2, 2k+1)$-torus knot) 
satisfies a first order relation
\begin{equation}
\Phi(N) = - q^{(k+\frac{1}{2})(1 - 2N)} \Phi(N-1)
+ q^{(k+\frac{1}{2})(1-N)} 
\frac{q^{\frac{2N-1}{2}}-q^{-\frac{2N-1}{2}}}{q^{\frac{1}{2}}-q^{-\frac{1}{2}}}
\end{equation}
instead of \eqref{eq:WRT-function-first-order-raletion} 
(c.f., \cite[eq.\,(9)]{Hikami-difference}). 
Thus, we may verify that the WRT function satisfies a second order 
$q$-difference equation 
\begin{align} 
\left[ \ql^2 + 
\biggl( q^{- 3( k + \frac{1}{2})} \qm^{-2(2k+1)}
 -q^{- (k + \frac{1}{2})} \, 
\frac{C'(q^{\frac{3}{2}} \qm)}{C'(q^{\frac{1}{2}} \qm)}
\biggr) \ql 
-  \qm^{-2(2k+1)}
\frac{C'(q^{\frac{3}{2}} \qm)}{C'(q^{\frac{1}{2}} \qm)}
 \right] \Phi=0, 
 \label{eq:q-difference-equation-WRT-degenerate}
\end{align}
where we have set $C'({\mathfrak m}) := {\mathfrak m} - {\mathfrak m}^{-1}$.
\end{rem}

%%%%%%%%%%%%%%%%%%%%%%%%%%%%%%%%%%%%%%%%%%%%%%%%%%%%%%%%%%%%%%%%%%%%%%%%%%%%%
\subsection{Classical limit and $A$-polynomial}

Let $\hat{A}(\hat{\mathfrak m}, \hat{\mathfrak l}; q)$ be
the $q$-difference operator, appearing in \eqref{eq:q-difference-equation-WRT}, 
which annihilates the WRT function $\Phi(q;N)$.
The classical limit of the $q$-difference operator 
$\hat{A}(\hat{\mathfrak m}, \hat{\mathfrak l}; q)$ is the algebraic curve 
in ${\mathbb C}^\ast \times {\mathbb C}^\ast$ defined by the equation 
$\hat{A}({\mathfrak m}, {\mathfrak l}; q=1) = 0$. More explicitly,
\begin{equation} 
\label{eq:A-polynomial}
{\mathfrak l}^3-{\mathfrak l}^2 -{\mathfrak m}^{-2P}{\mathfrak l} + {\mathfrak m}^{-2P}
=({\mathfrak l}-1)({\mathfrak l}-{\mathfrak m}^{-P})({\mathfrak l}+{\mathfrak m}^{-P})=0.
\end{equation}
Here $({\mathfrak m}, {\mathfrak l})$ is a coordinate of 
${\mathbb C}^\ast \times {\mathbb C}^\ast$. 
On the other hand, the classical limit of 
\eqref{eq:q-difference-equation-WRT-degenerate}, 
which corresponds to a degenerate situation 
$n=2$ and $(p_1, p_2)=(2,2k+1)$, is given by 
\begin{equation} 
\label{eq:A-polynomial-degenerate}
{\mathfrak l}^2 + ({\mathfrak m}^{-2(2k+1)}-1) {\mathfrak l} 
- {\mathfrak m}^{-2(2k+1)}  = 
({\mathfrak l}-1)({\mathfrak l} + {\mathfrak m}^{-2(2k+1)}) = 0.
\end{equation} 
Here we show that the following statement, 
which is closely related to the AJ-conjecture 
\cite{Garoufalidis, Gukov04} for a knot in $S^3$. 

\begin{thm}\label{AJ}
The classical limit \eqref{eq:A-polynomial} 
(resp., \eqref{eq:A-polynomial-degenerate}) 
of the $q$-difference equation 
\eqref{eq:q-difference-equation-WRT}
(resp., \eqref{eq:q-difference-equation-WRT-degenerate})
is a component of the zero locus of the $A$-polynomial of the Seifert loop $\sl$.
\end{thm}

\begin{proof}
The case $n = 2$ (including the degenerate case 
\eqref{eq:q-difference-equation-WRT-degenerate}) 
was proved by Hikami (\cite{Hikami-difference}). 
In what follows, we consider the case $n \ge 3$.

As in \cite[Section 7]{Beasley09}, 
the fundamental group 
$ \pi_1 (M \setminus L) $
of the Seifert loop 
is generated by
\begin{equation} \label{eq:generator-pi1}
c_j ~(j = 1, \dots, n),  ~
\mu, ~
f
\end{equation}
with the following relations:
\begin{align} 
\label{eq:relation-pi1-1}
& c_j^{p_j} f^{q_j} = 1, \\
\label{eq:relation-pi1-2}
& [\mu, f] = 1, \\
\label{eq:relation-pi1-3}
& \prod_{j=1}^{n} c_j = \mu.
\end{align}
Here, $c_j$'s correspond to small one-cycles around 
each of the orbifold points on the base sphere,
$\mu$ corresponds to the meridian element of $L$, 
and $f$ corresponds to a generic fiber which is represented by $L$.

Let $\lambda \in \pi_1(M \setminus L)$ be 
the longitude element of $L$. 
First we show the relation 
\begin{equation} \label{eq:relation-character-variety}
\lambda = f \mu^P
\end{equation}
of the fundamental group of the Seifert loop as follows. 

We find that there exists some integer $t$ such that 
\begin{equation} 
\lambda = f \mu^t
\end{equation}
because the longitude element $\lambda$ 
differs from the fiber element $f$ only with the framing. 
Then we have the relation of the homology group  
\begin{equation} \label{eq:a-priori-relation}
0=[f]+t[\mu]
\end{equation}
because $[\lambda]=0$ by the definition of the longitude element.
Here we have used the symbol $[\bullet]$ for the element in the homology group 
which is represented by $\bullet$. 
On the other hand, the relations 
\eqref{eq:relation-pi1-1} and \eqref{eq:relation-pi1-3} gives the relations
\begin{align}
& p_j[c_j]+q_j[f]=0 \label{eq:relation-homology-1}   \\
& \sum_{j=1}^{n} [c_j]=[\mu] \label{eq:relation-homology-2}  
\end{align}
in the homology group of the Seifert loop.
Therefore we have
\begin{equation} \label{eq:compute-homology-relation-Seifert}
P[\mu]=P \sum_{j=1}^{n} [c_j] 
= - P \, \sum_{j=1}^{n} \frac{q_j }{p_j}  \, [f] 
= - [f].  
\end{equation}
Here, we use \eqref{eq:relation-homology-2} in the first equality, 
\eqref{eq:relation-homology-1} in the second equality, 
and \eqref{eq:integral-homology-condition} 
in the last equality.
Thus we can conclude that the number $t$ in \eqref{eq:a-priori-relation} 
is equal to $P$, and hence, we have proved the relation 
\eqref{eq:relation-character-variety}. 

Second, let us consider an irreducible $2$-dimensional representation $\rho$ 
of the fundamental group of the Seifert loop. 
According to a discussion in \cite{A}, 
we have two irreducible representations $\rho_{+}$ and $\rho_{-}$ satisfying 
$\rho_{+}(f)={\rm Id}$ and $\rho_{-}(f)=-{\rm Id}$, respectively.
Then, it follows from \eqref{eq:relation-character-variety} that 
\begin{equation}
\rho_{+}(\lambda)=\rho_{+}(\mu)^P, \quad 
\rho_{-}(\lambda)=-\rho_{-}(\mu)^P
\end{equation} 
hold for these irreducible representations. 
Therefore, if we denote by ${\mathfrak m}$ (resp., ${\mathfrak l}$) 
one of the eigenvalues of the representation along the meridian $\mu$
(resp., longitude $\lambda$)  of the boundary torus of $M \setminus L$, 
we get the relation ${\mathfrak l}={\mathfrak m}^P$ for $\rho_{+}$, 
and ${\mathfrak l}=-{\mathfrak m}^P$ for $\rho_{-}$. 
Adding the contribution ${\mathfrak l}-1$ from the reducible part, we find the relation  
\begin{equation}
({\mathfrak l}-1)({\mathfrak l}-{\mathfrak m}^P)({\mathfrak l}+{\mathfrak m}^P)=0
\end{equation}
holds. This completes the proof of Theorem \ref{AJ}.
\end{proof}

%%%%%%%%%%%%%%%%%%%%%%%%%%%%%%%%%%%%%%%%%%%%%%%%%%%%%%%%%%%%%%%%%%%%%%%%%%%%%
%%%%%%%%%%%%%%%%%%%%%%%%%%%%%%%%%%%%%%%%%%%%%%%%%%%%%%%%%%%%%%%%%%%%%%%%%%%%%
\section{Resurgent analysis}
\label{section:resurgence}

In this section, for any fixed $N \in {\mathbb Z}_{\ge 1}$, 
we show that the WRT function $\Phi(q; N)$ 
is obtained as the average of the Borel sums (that is, the median sum) 
of the perturbative part $Z^{\rm pert}$ of the WRT invariant. 
We will employ the idea of Costin--Garoufalidis \cite{CG11} 
and Gukov--Marin\~o--Putrov \cite{GMP16}; see also \cite{Chun17, Chung18}. 
We refer \cite{Costin, Sauzin} for the fundamental facts on 
the Borel summation method and the resurgent analysis. 

\subsection{Perturbative part and its Borel transform}

First, for any fixed $N \in {\mathbb Z}_{\ge 1}$, 
let us introduce a function on $\{\kappa \in {\mathbb C}~;~{\rm Re}\, \kappa > 0 \}$ 
given by 
\begin{equation} \label{eq:trivial-connection-part-kappa}
Z_{\rm triv}(\kappa) := 
\frac{B}{2 \pi i} \, e^{- \frac{\pi i}{2\kappa} 
\bigl( \Theta_0 + (N^2 - 1) P \bigr)} I(\kappa), 
\end{equation} 
where 
\begin{equation} \label{eq:integral-part-kappa}
I(\kappa) := 
\int_{{\mathbb R}  e^{\frac{\pi i}{4}}} 
e^{\kappa g(y)} F_N(y) \, dy.
\end{equation}
(See Section \ref{section:WRT-invariant} for the definitions of $B$, $g(y)$ and $F_N(y)$. 
We drop the $N$-dependences from arguments since it is not relevant in this section.)
Note that, when we restrict $\kappa = K \in {\mathbb Z}_{\ge 1}$, 
the function $Z_{\rm triv}(\kappa)$ coincides with the one defined 
in \eqref{eq:trivial-connection-part} 
(i.e., the contribution from trivial connection to the WRT invariant).
In other words, the integer valued parameter $K$ in \eqref{eq:trivial-connection-part} 
is upgraded as a complex valued parameter $\kappa$ in \eqref{eq:trivial-connection-part-kappa}.

Let us define the formal power series
\begin{equation}
Z^{\rm pert}(\kappa) := \sum_{m=0}^{\infty} a_m \kappa^{-m - \frac{1}{2}} 
\in \kappa^{-\frac{1}{2}} \, {\mathbb C}[\hspace{-.13em}[\kappa^{-1}]\hspace{-.13em}], 
\end{equation}
which we call the perturbative part of the WRT invariant, 
as the asymptotic expansion of $Z_{\rm triv}(\kappa)$  
when $\kappa$ tends to $+\infty$ along the positive real axis: 
\begin{equation}
Z_{\rm triv}(\kappa) \sim Z^{\rm pert}(\kappa), ~~ \kappa \to + \infty.
\end{equation}
It follows from the definition that $Z^{\rm pert}(\kappa)$ admits a factorization
\begin{equation} \label{eq:Z-pert-expression}
Z^{\rm pert}(\kappa) = 
E(\kappa) \, I^{\rm pert}(\kappa),
\end{equation}
where 
$E(\kappa) \in {\mathbb C}[\hspace{-.13em}[\kappa^{-1}]\hspace{-.13em}]$ 
is the (convergent) Taylor series of obtained by expanding 
$$\frac{B}{2 \pi i} \,  e^{- \frac{\pi i}{2\kappa} 
\bigl( \Theta_0 + (N^2 - 1) P \bigr)}$$
when $\kappa \to \infty$, and 
$I^{\rm pert}(\kappa) \in \kappa^{-\frac{1}{2}} \, 
{\mathbb C}[\hspace{-.13em}[\kappa^{-1}]\hspace{-.13em}]$ 
is the asymptotic expansion of $I(\kappa)$  when $\kappa \to +\infty$.

The Borel transform of $Z^{\rm pert}(\kappa)$ is defined as follows:
\begin{equation}
Z_{B}^{\rm pert}(\xi) := \sum_{m = 0}^{\infty} 
\frac{a_m}{\Gamma(m  + \frac{1}{2})} \xi^{m - \frac{1}{2}},
\end{equation}
where $\xi$ is the Borel-Laplace dual variable to $\kappa$. 
Here, after taking the branch cut along the positive imaginary axis 
on $\xi$-plane, we take the principal branch of $\xi^{\frac{1}{2}}$
(i.e., $\xi^{\frac{1}{2}} \in {\mathbb R}_{> 0}$ 
when $\xi \in {\mathbb R}_{> 0}$). 

\begin{lem} \label{lem:Borel-transform}
The Borel transform $Z_{B}^{\rm pert}(\xi)$
converges on a punctured neighborhood of $\xi=0$, 
and is explicitly given by
\begin{equation} \label{eq:Borel-transform-Z-pert}
Z_{B}^{\rm pert}(\xi) 
= \bigl( E_{B} \ast 
I^{\rm pert}_{B} \bigr)(\xi).
\end{equation}
Here, 
$E_{B}(\xi)$ and $I_{B}^{\rm pert}(\xi)$
are the Borel transforms of $E(\kappa)$ 
and $I^{\rm pert}(\kappa)$, 
respectively, and $\ast$ is the convolution product: 
\begin{equation}
(f_1 \ast f_2) (\xi) := \int^{\xi}_{0} f_1(\eta) f_2(\xi - \eta) \,d\eta.
\end{equation}
Moreover, $I^{\rm pert}_B(\xi)$ is explicitly given as follows:
\begin{equation} 
I_{B}^{\rm pert}(\xi) 
= 
\left[ \frac{4\pi i P}{y} F_N (y)  
\right]_{y=\sqrt{8 \pi i P \xi}} 
-  
\left[ \frac{4\pi i P}{y} F_N(y)  
\right]_{y=-\sqrt{8 \pi i P \xi}}.
\label{eq:Borel-transform-I-pert}
\end{equation}
\end{lem}

\begin{proof}
Changing of integration variable by
\begin{equation} \label{eq:coordinate-change}
y \mapsto \sqrt{8 \pi i P \xi }, 
\end{equation}
the integral \eqref{eq:integral-part-kappa} is converted to 
a Laplace-type integral: 
\begin{align} 
I(\kappa) 
& = 
\int_{{\mathbb R}_{\ge 0}} 
e^{ - \kappa  \xi} \left[ \frac{4\pi i P}{y} F_N(y)  \right]_{y=\sqrt{8 \pi i P \xi}}
d\xi 
\notag \\[+.2em]
& \qquad 
- \int_{{\mathbb R}_{\ge 0}} 
e^{ - \kappa  \xi} \left[ \frac{4\pi i P}{y} F_N(y)  \right]_{y=-\sqrt{8 \pi i P \xi}}
d\xi.
\label{eq:usual-Borel-sum}
\end{align}
Then, since $I(\kappa) \sim I^{\rm pert}(\kappa)$ when $\kappa \to +\infty$, 
the Watson's lemma (see \cite[Chapter 4]{BH} for example)
on asymptotic expansion of Laplace-type integrals shows 
the equality \eqref{eq:Borel-transform-I-pert}.
Since the Borel transform converts a product of formal power series 
to the convolution product of the Borel transformed series, 
we obtain \eqref{eq:Borel-transform-Z-pert}.
\end{proof}

\subsection{Borel sum and Stokes automorphism}

Since $E(\kappa)$ is a convergent series, 
its Borel transform $E_{B}(\xi)$  is an entire function of $\xi$.
Hence, the above lemma shows that $Z^{\rm pert}_B(\xi)$ only has 
singularities along the positive imaginary axis in $\xi$-plane. 
Therefore, the formal series $Z^{\rm pert}(\kappa)$ 
is Borel summable in any direction $\theta$ satisfying 
\begin{equation} \label{eq:Stokes-phase}
\theta \equiv \hspace{-.9em} / \hspace{+.5em} \frac{\pi}{2} 
~ {\rm mod} ~ 2 \pi {\mathbb Z}. 
\end{equation}
Namely, for any such $\theta$, the integral 
\begin{equation} \label{eq:Borel-sum-in-the-direction-theta}
{\mathcal S}_\theta Z^{\rm pert}(\kappa) := 
\int_{{\mathbb R}_{\ge 0} e^{i \theta}} e^{- \kappa \xi} Z^{\rm pert}_{B}(\xi) \, d\xi
\end{equation}
converges on any closed sector contained in 
$\{\kappa \in {\mathbb C} ~;~ |\arg \kappa + \theta| < \frac{\pi}{2} \}$. 
The integral \eqref{eq:Borel-sum-in-the-direction-theta} is called the 
Borel sum of $Z^{\rm pert}(\kappa)$ in the direction $\theta$. 
Since the Laplace transform converts a convolution product 
to a usual product, we may verify that 
$Z_{\rm triv}(\kappa)$ coincides with the Borel sum 
of $Z^{\rm pert}(\kappa)$ in the direction $0$.

In what follows, we set 
\begin{equation}
\theta_0 := \frac{\pi}{2},
\end{equation}
and 
\begin{equation}
{\mathcal S}^{\pm}_{\theta_0}Z^{\rm pert}(\kappa) 
:= {\mathcal S}_{\theta_0 \pm \delta}Z^{\rm pert}(\kappa)
\end{equation}
for a sufficiently small $\delta > 0$. 
We will take $\delta$ smaller if necessary.
Since $Z^{\rm pert}_{B}(\xi)$ has singularities on ${\mathbb R}_{> 0} e^{i \theta_0}$, 
the Borel sums 
${\mathcal S}^{+}_{\theta_0} Z^{\rm pert}(\kappa)$ and 
${\mathcal S}^{-}_{\theta_0} Z^{\rm pert}(\kappa)$ do not agree
on the sector 
$\{\kappa \in {\mathbb C} ~;~ |\arg \kappa + \theta_0| < \frac{\pi}{2} - \delta \}$.
This is nothing but the Stokes phenomenon. 
According to the general theory of resurgent analysis, 
such a difference is described by 
the alien derivatives and the Stokes automorphism. 
Here we briefly recall these notions. (See \cite{Sauzin} for details). 

First, Lemma \ref{lem:Borel-transform} shows that the formal series 
$Z^{\rm pert}(\kappa)$ has the Borel transform 
$Z_B^{\rm pert}(\xi)$ which has only simple singularities 
(in the sense of \cite[Section 26]{Sauzin}) along the half line 
${\mathbb R}_{> 0} e^{i \theta_0}$. 
A formal power series is said to be simple resurgent 
if its Borel transform only has simple singularities on a certain discrete set. 

Let 
\begin{equation}
\Omega \subset \left\{ \frac{\pi i m^2}{2P} ~;~ m \in {\mathbb Z}_{\ge 1}  \right\} 
\end{equation}
be the set of the simple singularities of $Z^{\rm pert}_B(\xi)$ 
on the half line ${\mathbb R}_{> 0} e^{i \theta_0}$.  
As is proved in \cite{AP, Mistegard}, the location of the poles are closely related to 
the complex Chern--Simons values of the flat connections on $M$.
For each $\omega \in \Omega$, the alien derivative 
$\Delta_\omega$ at $\omega$ is defined as an operator acting on 
the space of simple resurgent formal power series. 
We do not give the definition of the alien derivatives here, 
but let us summarize several properties 
which are relevant for our purpose 
(See \cite[Section 28]{Sauzin} for the definition):
\begin{itemize}
\item $\Delta_\omega$ is a derivation: 
$\Delta_\omega (f \, g) = (\Delta_\omega f) \, g + f \, (\Delta_\omega g)$
(see \cite[Section 30]{Sauzin}). 
\item 
Convergent series are annihilated by $\Delta_\omega$ 
(see \cite[Section 27]{Sauzin}).
\item 
Since $I_{B}^{\rm pert}(\xi)$ has only simple poles at $\omega \in \Omega$, 
we have
\begin{equation}
\Delta_\omega I^{\rm pert}(\kappa) = 2 \pi i \, 
\Res_{\xi = \omega} I^{\rm pert}_{B}(\xi) d\xi
\end{equation}
(see \cite[Example 27.4]{Sauzin}). 
\end{itemize}

\smallskip
\begin{lem}  \label{lem:action-of-alien-derivatives}
\begin{itemize}
\item[(i)] 
For any $\omega \in \Omega$, we have
\begin{equation} \label{eq:alien-action-Z-pert}
\Delta_{\omega} Z^{\rm pert}(\kappa) = 2 \pi i \, 
\Bigl( \Res_{\xi = \omega} I^{\rm pert}_{B}(\xi) d\xi \Bigr) \, E(\kappa).
\end{equation}

\item[(ii)]
Iterative actions of alien derivatives annihilate $Z^{\rm pert}$. 
That is, for any $\omega_1, \dots, \omega_r \in \Omega$ with $r \ge 2$, we have
\begin{equation}
\Delta_{\omega_1} \cdots \Delta_{\omega_r} Z^{\rm pert}(\kappa) = 0. 
\end{equation}
\end{itemize}
\end{lem}

\begin{proof}
Using the above properties of alien derivatives, 
we have
\begin{align}
\Delta_{\omega} Z^{\rm pert}(\kappa) & = 
(\Delta_{\omega} E(\kappa)) \, I^{\rm pert}(\kappa) + 
E(\kappa) \, (\Delta_{\omega} I^{\rm pert}(\kappa))  \notag \\
& = 2 \pi i \, \Bigl( \Res_{\xi = \omega} I^{\rm pert}_{B}(\xi) d\xi \Bigr) \, E(\kappa).
\end{align}
(Note that $\Delta_{\omega} E(\kappa) = 0$ since $E(\kappa)$ is a convergent series 
of $\kappa^{-1}$.) 
This proves (i). The second claim (ii) immediately follows from (i) 
since the \eqref{eq:alien-action-Z-pert} is a convergent series.
\end{proof}

Using the alien derivatives, the Stokes automorphism for the direction $\theta_0$
is defined as follows:
\begin{equation} \label{eq:Stokes-auto}
{\mathfrak S}_{\theta_0} 
= \exp\left( \sum_{\omega \in \Omega} e^{- \kappa \omega} \Delta_{\omega} \right).
\end{equation}
This is defined so that 
\begin{equation} \label{eq:Stokes-auto-and-Borel-sums}
{\mathcal S}^{-}_{\theta_0} Z^{\rm pert}(\kappa) 
= 
{\mathcal S}^{+}_{\theta_0} 
{\mathfrak S}_{\theta_0}
Z^{\rm pert}(\kappa)
\end{equation} 
holds on any closed sector included in the lower half plane 
$\{\kappa \in {\mathbb C} ~;~ {\rm Im} \, \kappa < 0 \}$. 
(c.f., \cite[Section 29]{Sauzin}).

\subsection{Median sum}

To formulate our main claim in this section, 
let us introduce the notion of the median summation (c.f., \cite{DP99}).
The median sum of $Z^{\rm pert}(\kappa)$ in the direction $\theta_0$ is defined by
\begin{equation} \label{eq:def-median-sum}
{\mathcal S}^{\rm med}_{\theta_0} Z^{\rm pert}(\kappa) := 
{\mathcal S}^{+}_{\theta_0} \circ 
{\mathfrak S}_{\theta_0}^{+\frac{1}{2}} Z^{\rm pert}(\kappa)
=
{\mathcal S}^{-}_{\theta_0} \circ 
{\mathfrak S}_{\theta_0}^{-\frac{1}{2}} Z^{\rm pert}(\kappa),
\end{equation}
where ${\mathfrak S}_{\theta_0}^{\pm \frac{1}{2}}$ 
is defined by 
\begin{equation} \label{eq:Stokes-auto-half}
{\mathfrak S}_{\theta_0}^{\pm \frac{1}{2}} := 
\exp\left( \pm \frac{1}{2} \sum_{\omega \in \Omega} 
e^{- \kappa \omega} \Delta_{\omega} \right).
\end{equation}
Note that, although the usual Borel sum \eqref{eq:Borel-sum-in-the-direction-theta} 
cannot be defined for $\theta = \theta_0$, the median sum is well-defined. 
The median summation is important since it transforms formal power series 
with real coefficients into real analytic functions of $\kappa$, if it converges
(c.f., \cite[p.21]{DP99}).

\begin{lem} \label{lem:median-sum-and-average}
The median sum of $Z^{\rm pert}(\kappa)$ in the direction $\theta_0$ 
is expressed as an average of Borel sums
\begin{equation}
{\mathcal S}^{\rm med}_{\theta_0} Z^{\rm pert}(\kappa) = 
\frac{{\mathcal S}^{+}_{\theta_0} \, Z^{\rm pert}(\kappa) 
+ {\mathcal S}_{\theta_0}^{-} Z^{\rm pert}(\kappa)}{2}
\end{equation}
on any closed sector included in the lower half plane 
$\{\kappa \in {\mathbb C} ~;~ {\rm Im} \, \kappa < 0 \}$.
\end{lem}

\begin{proof}
Lemma \ref{lem:action-of-alien-derivatives} (ii) shows 
\begin{align}
{\mathfrak S}_{\theta_0} Z^{\rm pert}(\kappa)
& = Z^{\rm pert}(\kappa) + 
\left(\sum_{\omega \in \Omega} e^{- \kappa \omega} \Delta_{\omega}\right) Z^{\rm pert}(\kappa), 
\label{eq:action-Stokes-auto}
\\
{\mathfrak S}_{\theta_0}^{\pm \frac{1}{2}}Z^{\rm pert}(\kappa) 
& = Z^{\rm pert}(\kappa) \pm \frac{1}{2} 
\left(\sum_{\omega \in \Omega} e^{- \kappa \omega} \Delta_{\omega} \right) Z^{\rm pert}(\kappa).
\label{eq:action-Stokes-auto-half}
\end{align}
In particular, we have 
\begin{equation}
{\mathfrak S}_{\theta_0}^{\pm \frac{1}{2}}Z^{\rm pert}(\kappa) 
= Z^{\rm pert}(\kappa) \pm \frac{1}{2}  
\left( 
{\mathfrak S}_{\theta_0} Z^{\rm pert}(\kappa) 
- Z^{\rm pert}(\kappa)
\right).
\end{equation}
Then, using \eqref{eq:Stokes-auto-and-Borel-sums}, we obtain
\begin{align}
{\mathcal S}^{\rm med}_{\theta_0} Z^{\rm pert}(\kappa) 
& = {\mathcal S}^{+}_{\theta_0} {\mathfrak S}_{\theta_0}^{+\frac{1}{2}} Z^{\rm pert}(\kappa) \notag \\
& = {\mathcal S}^{+}_{\theta_0} Z^{\rm pert}(\kappa)
+ \frac{1}{2} \left( 
{\mathcal S}^{+}_{\theta_0}{\mathfrak S}_{\theta_0} Z^{\rm pert}(\kappa) 
- {\mathcal S}^{+}_{\theta_0} Z^{\rm pert}(\kappa)
\right) \notag \\
& = \frac{1}{2} \left( {\mathcal S}^{+}_{\theta_0} Z^{\rm pert}(\kappa) + 
{\mathcal S}^{-}_{\theta_0} Z^{\rm pert}(\kappa)  \right).
\end{align}
This completes the proof of Lemma \ref{lem:median-sum-and-average}. 
\end{proof}

\begin{rem}
Both of the Borel sums ${\mathcal S}_0 Z^{\rm pert}(\kappa)$
and ${\mathcal S}_\pi Z^{\rm pert}(\kappa)$, which are defined by 
\eqref{eq:Borel-sum-in-the-direction-theta} for $\theta = 0$ and $\pi$, respectively, 
have an analytic continuation to a sector 
\[
\left\{ \kappa \in {\mathbb C} ~;~ \left| \arg \kappa + \frac{\pi}{2} \right| 
\le  \frac{\pi}{2} - \delta \right\}
\]
for any fixed $\delta > 0$. More precisely, 
since the Borel transform $Z^{\rm pert}_B(\xi)$ only has 
singularities along the positive imaginary axis, 
the analytic continuations of ${\mathcal S}_0 Z^{\rm pert}(\kappa)$
and ${\mathcal S}_\pi Z^{\rm pert}(\kappa)$ to the above sector
are explicitly given by 
${\mathcal S}_{\theta_0}^{-} Z^{\rm pert}(\kappa)$ and 
${\mathcal S}_{\theta_0}^{+} Z^{\rm pert}(\kappa)$, respectively. 
Therefore, we may write the median sum as the average 
of ${\mathcal S}_0 Z^{\rm pert}(\kappa)$ and 
${\mathcal S}_\pi Z^{\rm pert}(\kappa)$ on the above sector. 
(C.f., \cite{CG11, GMP16}),
\end{rem}

\subsection{WRT function as median sum}

The main claim of this section is the following.
\begin{thm} \label{thm:WRT-function-as-Borel-sum}
Under the change $q = \exp(\frac{2 \pi i}{\kappa})$ of the variables, 
the WRT function $\Phi(q)$ is expressed as 
\begin{equation} \label{eq:WRT-function-as-Borel-sum}
\frac{\Phi(e^{\frac{2 \pi i}{\kappa}})}{G_0(\kappa)} = 
{\mathcal S}^{\rm med}_{\theta_0} Z^{\rm pert}(\kappa).
\end{equation}
on any closed sector included in the lower half plane 
$\{\kappa \in {\mathbb C} ~;~ {\rm Im} \, \kappa < 0 \}$.
\end{thm}

\begin{proof}
We will prove \eqref{eq:WRT-function-as-Borel-sum}
by using the techniques developed in \cite{CG11, GMP16}.

It follows from Lemma \ref{lem:Borel-transform} that the Borel sum 
\eqref{eq:Borel-sum-in-the-direction-theta} of $Z^{\rm pert}(\kappa)$ 
in any direction $\theta \ne \theta_0$ is given by 
\begin{equation}
{\mathcal S}_{\theta} Z^{\rm pert}(\kappa) = 
\frac{B}{2 \pi i } \, 
e^{- \frac{\pi i}{2\kappa} 
\bigl( \Theta_0 + (N^2 - 1) P \bigr)} \,
\int_{{\mathbb R} \, e^{i (\frac{\pi }{4} + \frac{\theta}{2})}} 
e^{\kappa g(y)} F_N(y) \,  dy,
\end{equation}
where the integration contour is 
oriented from $(- \infty) \cdot e^{i (\frac{\pi }{4} + \frac{\theta}{2})}$ 
to $(+ \infty) \cdot e^{i (\frac{\pi }{4} + \frac{\theta}{2})}$.
Then Lemma \ref{lem:median-sum-and-average} shows that the right-hand side of 
\eqref{eq:WRT-function-as-Borel-sum} is computed as follows: 
\begin{align} \label{eq:average-Borel-sum-on-y-palne}
& \frac{1}{2} \left( {\mathcal S}_{\theta_0}^{+} Z^{\rm pert}(\kappa) 
+ {\mathcal S}_{\theta_0}^{-} Z^{\rm pert}(\kappa)  \right) 
\notag \\
& \quad = \frac{B}{4 \pi i } \, 
e^{- \frac{\pi i}{2\kappa} 
\bigl( \Theta_0 + (N^2 - 1) P \bigr)} \,
\left( 
\int_{{\mathbb R} \, e^{i ( \frac{\pi}{2} + \frac{\delta}{2} ) }} 
+ 
\int_{{\mathbb R} \, e^{i ( \frac{\pi}{2} - \frac{\delta}{2} ) }}
\right) 
 e^{\kappa g(y)} F_N(y) \, dy 
\notag \\
& \quad = 
\frac{B}{4 \pi i } \, 
e^{- \frac{\pi i}{2\kappa} 
\bigl( \Theta_0 + (N^2 - 1) P \bigr)} \, \left( 
\int_{y =  + \epsilon + i {\mathbb R}} 
+
\int_{y =  - \epsilon + i {\mathbb R}} 
\right) 
e^{\kappa g(y)} F_N(y) dy
\notag \\
& \quad = 
\frac{B}{2 \pi i } \, 
e^{- \frac{\pi i}{2\kappa} 
\bigl( \Theta_0 + (N^2 - 1) P \bigr)} \, 
\int_{y =  + \epsilon + i {\mathbb R}} 
e^{\kappa g(y)} F_N(y) dy. 
\end{align}
Here $\epsilon$ is any positive number, 
and we have used the fact that the integrant 
$e^{\kappa g(y)} F_N(y)$
is invariant under $y \mapsto -y$. 
Since $| e^{-y} | < 1$ holds on $+ \epsilon + i {\mathbb R}$, 
we have a uniformly convergent expression of the integrand as
\begin{align}
e^{\kappa g(y)} F_N(y) 
& = (-1)^n  
\sum_{\ell = - \frac{N-1}{2}}^{\frac{N-1}{2}} 
\sum_{(\varepsilon_1, \dots, \varepsilon_n)} 
\varepsilon_1 \cdots \varepsilon_n 
\notag \\
& \quad \times
\sum_{m = 0}^{\infty} \dbinom{m+n-3}{n-3} \, 
e^{-\bigl(2m + 2\ell + n-2 + \sum_{j=1}^{n} \frac{\varepsilon_j}{p_j} \bigr) 
\frac{y}{2} + \frac{i \kappa}{8 \pi P} y^2}.
\end{align}
Evaluating the integral \eqref{eq:average-Borel-sum-on-y-palne} 
term by term by using 
\[
\int_{+ \epsilon + i {\mathbb R}} 
e^{-\bigl(2m + 2\ell + n-2 + \sum_{j=1}^{n} \frac{\varepsilon_j}{p_j} \bigr) 
\frac{y}{2} + \frac{i \kappa}{8 \pi P} y^2} dy 
= 4 \pi e^{\frac{\pi i}{4}} \sqrt{\frac{P}{2 \kappa}} \,
e^{\frac{2 \pi i P}{4\kappa} 
\bigl( 2m + 2\ell + n-2 + \sum_{j=1}^{n} \frac{\varepsilon_j}{p_j} \bigr)^2},
\]
we obtain the equality \eqref{eq:WRT-function-as-Borel-sum}.
\end{proof}

%%%%%%%%%%%%%%%%%%%%%%%%%%%%%%%%%%%%%%%%%%%%%%%%%%%%%%%%%%%%%%%%%%%%%%%%%%%%%
%%%%%%%%%%%%%%%%%%%%%%%%%%%%%%%%%%%%%%%%%%%%%%%%%%%%%%%%%%%%%%%%%%%%%%%%%%%%%
\appendix

%%%%%%%%%%%%%%%%%%%%%%%%%%%%%%%%%%%%%%%%%%%%%%%%%%%%%%%%%%%%%%%%%%%%%%%%%%%%%
%%%%%%%%%%%%%%%%%%%%%%%%%%%%%%%%%%%%%%%%%%%%%%%%%%%%%%%%%%%%%%%%%%%%%%%%%%%%%
\section{Proof of Lemma \ref{lemm:vanishing-sums}}
\label{appendix:vanishing-sums}

Here we give a proof of Lemma \ref{lemm:vanishing-sums}. 
We will prove the following stronger statement: 
For any $s \in \{0,1,\dots, n-1 \}$ and any $\ell \in \frac{1}{2} {\mathbb Z}$, 
we have
\begin{equation} \label{eq:vanishing-sums-appendix}
\sum_{(\varepsilon_1, \dots, \varepsilon_n) \in \{ \pm 1 \}^n} 
\varepsilon_1 \cdots \varepsilon_n \cdot
\biggl( \sum_{j=1}^{n} \frac{\varepsilon_j}{p_j} \biggr)^s \cdot
\sum_{m=0}^{K-1} e^{\frac{\pi i}{2PK}\, ( 2Pm + a_{\ell, \varepsilon} )^2}
= 0.
\end{equation}
(Recall that $a_{\ell, \varepsilon} = P(2 \ell + n-2 
+ \sum_{j=1}^{n}({\varepsilon_j}/{p_j})) \in {\mathbb Z}$.)

For the purpose, first we show 
\begin{lem}\label{lemma:exponents-of-gauss-sum}
For any $\ell \in \frac{1}{2} {\mathbb Z}$
and any  
$\varepsilon = (\varepsilon_1, \dots, \varepsilon_n)$, 
$\tilde{\varepsilon} = (\tilde{\varepsilon}_1, \dots, \tilde{\varepsilon}_n) 
\in \{ \pm 1\}^n$, there exists a bijection 
\begin{equation} \label{eq:bijection-eplison-tildeepsilon}
\varphi_{\varepsilon, \tilde{\varepsilon}} : 
\{ (2Pm+a_{\ell, \varepsilon})^2 ~;~ m \in \{0,1,\dots,K-1 \} \rightarrow 
\{ (2Pm+a_{\ell, \tilde{\varepsilon}})^2 ~;~ m \in \{0,1,\dots,K-1 \}
\end{equation} 
such that 
\begin{equation}
x \equiv \varphi_{\varepsilon, \tilde{\varepsilon}}(x) ~ {\rm mod} ~ 4KP
\end{equation}
holds for all $x \in \{ (2Pm+a_{\ell, \varepsilon})^2 ~;~ m \in \{0,1,\dots,K-1 \}$.
\end{lem}

\begin{proof}[Proof of Lemma \ref{lemma:exponents-of-gauss-sum}]
For $m \in {\mathbb Z}$, define 
$x(m) := (2Pm+a_{\ell, \varepsilon})^2$ and 
$\tilde{x}(m) := (2Pm+a_{\ell, \tilde{\varepsilon}})^2$. 
To construct a bijection $\varphi_{\varepsilon, \tilde{\varepsilon}}$, 
we shall prove the following:
There is a bijective correspondence 
$\{0,1,\dots,K-1 \} \ni m \mapsto \tilde{m} \in \{0,1,\dots,K-1 \}$
such that 
\begin{equation} \label{eq:x-tilde-x}
x(m) \equiv \tilde{x}(\tilde{m}) ~ {\rm mod} ~ 4 KP
\end{equation}
holds. Then we may define the desired map 
\eqref{eq:bijection-eplison-tildeepsilon} by 
$\varphi_{\varepsilon, \tilde{\varepsilon}}(x(m)) = \tilde{x}(\tilde{m})$.

It is enough to prove the statement in the case 
\begin{equation}
\varepsilon = (\varepsilon_1, \varepsilon_2, \dots, \varepsilon_n),\quad
\tilde{\varepsilon} = (- \varepsilon_1, \varepsilon_2, \dots, \varepsilon_n),
\end{equation}
with $\varepsilon_1, \varepsilon_2, \dots, \varepsilon_n \in \{ \pm 1 \}$
being chosen arbitrary. In this case, we have
\begin{align}
& x(m) - \tilde{x}(\tilde{m}) \notag \\
& \quad = 
4P \biggl( p_1(m-\tilde{m})  + {\varepsilon_1} \biggr)
\biggl(p_2\cdots p_n \bigl( m+\tilde{m} + 2\ell+n-2 \bigr) 
+ p_2 \cdots p_n \sum_{j=2}^{n} \frac{\varepsilon_j}{p_j} \biggr)
\end{align}
for any $m, \tilde{m} \in \{0,1,\dots,K-1 \}$.

\bigskip
$\bullet$ \underline{The case $\gcd(p_1, K) = 1$ 
or $\gcd(p_2 \cdots p_n, K) = 1$.} \, 
In the case $\gcd(p_1, K) = 1$, for any given $m$, 
we can always find $\tilde{m}$ such that 
$ p_1(m - \tilde{m}) + \varepsilon_1 \in K \mathbb Z$ by the Euclidean algorithm. 
This gives a desired mapping satisfying \eqref{eq:x-tilde-x} 
since the correspondence $m \mapsto \tilde{m}$ is bijective modulo $K$.  
We can obtain a bijective map when $\gcd(p_2 \cdots p_n, K) = 1$ 
by a similar manner.

\bigskip

$\bullet$ \underline{The case $\gcd(p_1, K) > 1$ and 
$\gcd(p_2 \cdots p_n, K) > 1$.} \, 
For $m, \tilde{m} \in \{0,1,\dots, K-1 \}$, we set 
\begin{align}
X^{(1)}(m,\tilde{m}) & := p_1 (m-\tilde{m}) + \varepsilon, \\
X^{(2)}(m,\tilde{m}) & := p_2 \cdots p_n (m+\tilde{m}+2\ell+n-2) 
+ p_2 \cdots p_n \sum_{j=2}^{n} \frac{\varepsilon_j}{p_j}.
\end{align}
We must look at the structure of reminders 
(obtained by dividing by factors of $K$) of both 
$X^{(1)}(m,\tilde{m})$ and $X^{(2)}(m,\tilde{m})$.
For the purpose, let us introduce several notations. 

Starting from 
\begin{equation}
G^{(1)}_1 := \gcd(p_1, K) > 1, \quad 
K^{(1)}_1 := {K}/{G^{(1)}_1},
\end{equation} 
we define integers $G^{(1)}_i, K^{(1)}_i$ 
for $i \ge 2$ inductively by 
\begin{equation}
G^{(1)}_i := \gcd(p_1, K_{i-1}^{(1)}), \quad 
K^{(1)}_i := {K_{i-1}^{(1)}}/{G^{(1)}_i}.
\end{equation} 
Similarly, starting from 
\begin{equation}
G^{(2)}_1 := \gcd(p_2\cdots p_n, K) > 1, \quad 
K^{(2)}_1 := {K}/{G^{(2)}_1},
\end{equation} 
define integers $G^{(2)}_i, K^{(2)}_i$ 
for $i \ge 2$ inductively by 
\begin{equation}
G^{(2)}_i := \gcd(p_2\cdots p_n, K_{i-1}^{(2)}), \quad 
K^{(2)}_i := {K_{i-1}^{(2)}}/{G^{(2)}_i}.
\end{equation} 
For $s \in \{1, 2\}$, 
since $G^{(s)}_{i} | G^{(s)}_{i-1}$, 
there exists unique $d^{(s)} \ge 1$ such that 
\begin{equation}
G^{(s)}_{d^{(s)}} \ne 1 \quad \text{and} \quad
G^{(s)}_{d^{(s)}+1} = G^{(s)}_{d^{(s)}+2} = \cdots = 1.
\end{equation}
It also follows from the definition that $K$ is expressed as 
\begin{equation}
K 
 = G^{(1)}_1 \cdots G^{(1)}_{d^{(1)}} \cdot K^{(1)}_{d^{(1)}} 
 = G^{(2)}_1 \cdots G^{(2)}_{d^{(2)}} \cdot K^{(2)}_{d^{(2)}},
\end{equation}
which in particular implies that 
\begin{equation}
K = { K^{(1)}_{d^{(1)}} \cdot K^{(2)}_{d^{(2)}} }/{f},
\end{equation}
where 
\begin{equation}
f := {\rm gcd}(K^{(1)}_{d^{(1)}}, K^{(2)}_{d^{(2)}}) 
= \frac{K^{(1)}_{d^{(1)}}}{G^{(2)}_1 \cdots G^{(2)}_{d^{(2)}}}
= \frac{K^{(2)}_{d^{(2)}}}{G^{(1)}_1 \cdots G^{(1)}_{d^{(1)}}}. 
\end{equation}
Here we have used the fact that $G^{(1)}_i$ and $G^{(2)}_j$ 
do not have any non-trivial common divisors due to 
the pair-wise coprimeness of $p_1, \dots, p_n$.
Finally, we define
\begin{equation}
K^{(1)} := K^{(1)}_{d^{(1)}}, \quad
K^{(2)} := K^{(2)}_{d^{(2)}}/f.
\end{equation} 
In summary, we have obtained a factorization 
$K = K^{(1)} \cdot K^{(2)}$ satisfying 
\begin{equation} \label{eq:key-gcd-properties}
\gcd(p_1, K^{(1)}) = 1, \quad
\gcd(p_2\cdots p_n, K^{(2)}) = 1, \quad
\gcd(K^{(1)}, K^{(2)}) = 1.
\end{equation}
The properties in \eqref{eq:key-gcd-properties} 
are essential in the rest of the proof.

\begin{lem} \label{lemma:congruence-relation}
For any fixed $m \in \{0,1,\dots, K-1 \}$, 
the system of congruence equations
\begin{equation} \label{eq:system-of-congruence-relation}
\begin{cases}
X^{(1)}(m,\tilde{m}) \equiv 0
\,\,{\rm mod}\,\, K^{(1)} \\[+.5em]
X^{(2)}(m,\tilde{m}) \equiv 0
\,\,{\rm mod}\,\, K^{(2)}
\end{cases}
\end{equation}
has a unique solution $\tilde{m} \in \{0,1,\dots, K-1 \}$.
\end{lem}
\begin{proof}[Proof of Lemma \ref{lemma:congruence-relation}]
Although the claim follows from the Chinese remainder theorem, 
let us give a proof here for the convenience of readers.  

Firstly, note that the coprimeness of $p_1$ and $K^{(1)}$ 
implies that the congruence relation 
$X^{(1)}(m,\tilde{m}_1) - X^{(1)}(m,\tilde{m}_2) = 
p_1 (\tilde{m}_1 - \tilde{m}_2) \equiv 0$ mod $K^{(1)}$
is satisfied if and only if 
$\tilde{m}_1 \equiv \tilde{m}_2$ mod $K^{(1)}$.
Therefore, we have 
\begin{equation} \label{eq:mod-X1-K1}
\{ X^{(1)}(m,\tilde{m})\,\,{\rm mod}\,\,K^{(1)} ~;~ 
\tilde{m} \in \{0,1,\dots, K^{(1)} - 1 \} \} 
= \{ 0,1,\dots, K^{(1)} - 1 \}
\end{equation}
because the left hand side is a subset of the right hand-side 
consisting of $K^{(1)}$ distinct elements. 
Similarly, we have
\begin{equation} \label{eq:mod-X2-K2}
\{ X^{(2)}(m,\tilde{m})\,\,{\rm mod}\,\,K^{(2)} ~;~ 
\tilde{m} \in \{0,1,\dots, K^{(2)} - 1 \} \} 
= \{ 0,1,\dots, K^{(2)} - 1 \}.
\end{equation}
It follows from \eqref{eq:mod-X1-K1} 
(resp., and \eqref{eq:mod-X2-K2}) that 
there exists a unique 
$\tilde{m}^{(1)}_0 \in \{0,1,\dots, K^{(1)} - 1 \}$
(resp., $\tilde{m}^{(2)}_0 \in \{0,1,\dots, K^{(2)} - 1 \}$)
satisfying 
\begin{equation}
X^{(1)}(m,\tilde{m}^{(1)}_0) \equiv 0 \,\,{\rm mod}\,\,K^{(1)} \quad 
\bigl({\rm resp}.,~ X^{(2)}(m,\tilde{m}^{(2)}_0) 
\equiv 0\,\,{\rm mod}\,\,K^{(2)} \bigr).
\end{equation}
Then, using the coprimeness of $K^{(1)}$ and $K^{(2)}$, 
we can find an integer $\tilde{m}$, 
which is unique modulo $K^{(1)} \cdot K^{(2)} = K$, 
satisfying $\tilde{m} \equiv \tilde{m}^{(s)}_0$ mod $K^{(s)}$ 
for both $s \in \{1, 2\}$ (by the Euclidean algorithm). 
It is easy to verify that the $\tilde{m}$ obtained here 
is the solution of \eqref{eq:system-of-congruence-relation}. 
This completes the proof of Lemma \ref{lemma:congruence-relation}.
\end{proof}

By Lemma \ref{lemma:congruence-relation}, 
we have a well-defined map 
$\{0,1,\dots, K-1 \} \ni m \mapsto \tilde{m} \in \{0,1,\dots, K-1 \}$ 
by solving \eqref{eq:system-of-congruence-relation}. 
The map satisfies \eqref{eq:x-tilde-x} since 
\begin{equation}
X^{(1)}(m,\tilde{m}) \cdot X^{(2)}(m,\tilde{m}) 
\in K^{(1)} {\mathbb Z} \cap K^{(2)} {\mathbb Z} = K {\mathbb Z}.
\end{equation}
To complete the proof of Lemma \ref{lemma:exponents-of-gauss-sum}, 
we must show that the correspondence between $m$ and $\tilde{m}$ is bijective.
This will be done as follows.

Suppose $m_1, m_2 \in \{0,1,\dots, K-1\}$ are mapped by 
$\tilde{m}_1, \tilde{m}_2 \in \{0,1,\dots, K-1\}$ by the correspondence 
specified by Lemma \ref{lemma:congruence-relation}, respectively.
In other words, suppose that 
\begin{equation}
\begin{cases}
X^{(1)}(m_1, \tilde{m}_1) \equiv X^{(1)}(m_2, \tilde{m}_2) \equiv 0 
\,\,{\rm mod}\,\, K^{(1)}, \\
X^{(2)}(m_1, \tilde{m}_1) \equiv X^{(2)}(m_2, \tilde{m}_2) \equiv 0
\,\,{\rm mod}\,\,K^{(2)} 
\end{cases}
\end{equation}
hold simultaneously. In particular, they satisfy
\begin{equation}
\begin{cases}
p_1(m_1 - \tilde{m}_1) \equiv p_1(m_2 - \tilde{m}_2) 
\,\,{\rm mod}\,\, K^{(1)}\\
p_2\cdots p_n(m_1 - \tilde{m}_1) \equiv p_2\cdots p_n(m_2 - \tilde{m}_2) 
\,\,{\rm mod}\,\, K^{(2)}.
\end{cases}
\end{equation}
The coprimeness \eqref{eq:key-gcd-properties} implies that
$(m_1 - m_2) \equiv (\tilde{m}_1 - \tilde{m}_2)$ mod $K$. 
Hence we have $\tilde{m}_1 \equiv \tilde{m}_2$ mod $K$ if and only if 
${m}_1 \equiv {m}_2$ mod $K$. Thus we have verified that 
the above correspondence $m \mapsto \tilde{m}$ is bijective.
This completes the proof of Lemma \ref{lemma:exponents-of-gauss-sum}.
\end{proof}

Thanks to Lemma \ref{lemma:exponents-of-gauss-sum}, we have
\begin{multline}
\sum_{(\varepsilon_1, \dots, \varepsilon_n) \in \{ \pm 1 \}^n} 
\varepsilon_1 \cdots \varepsilon_n \cdot
\biggl( \sum_{j=1}^{n} \frac{\varepsilon_j}{p_j} \biggr)^s \cdot
\sum_{m=0}^{K-1} e^{\frac{\pi i}{2PK}\, ( 2Pm + a_{\ell, \varepsilon} )^2} 
\\
= 
\left( \sum_{m=0}^{K-1} e^{\frac{\pi i}{2PK}\, ( 2Pm + a_{\ell, \varepsilon^{(0)}} )^2} \right)
\cdot  
\sum_{(\varepsilon_1, \dots, \varepsilon_n) \in \{ \pm 1 \}^n} 
\varepsilon_1 \cdots \varepsilon_n \cdot
\biggl( \sum_{j=1}^{n} \frac{\varepsilon_j}{p_j} \biggr)^s, 
\end{multline}
where $\varepsilon^{(0)} \in \{\pm 1 \}^n$ is any fixed $n$-tuple of signatures.
Then, the desired equality \eqref{eq:vanishing-sums-appendix} 
is reduced to the following simpler statement: 

\begin{lem}
For each $s \in \{0,1,\dots, n-1 \}$, we have 
\begin{equation} \label{eq:vanishing-sums-appendix-2}
\sum_{(\varepsilon_1, \dots, \varepsilon_n) \in \{ \pm 1 \}^n} 
\varepsilon_1 \cdots \varepsilon_n \cdot
\biggl( \sum_{j=1}^{n} \frac{\varepsilon_j}{p_j} \biggr)^s 
= 0.
\end{equation}
\end{lem}

\begin{proof}
The claim for $s=0$ is trivial. 
Since 
\begin{align}
& \sum_{(\varepsilon_1, \dots, \varepsilon_n)} 
\varepsilon_1 \cdots \varepsilon_n \cdot
\biggl( \sum_{j=1}^{n} \frac{\varepsilon_j}{p_j} \biggr)^s 
= \sum_{j=1}^{n} \frac{1}{p_j} 
\sum_{(\varepsilon_1, \dots, \varepsilon_n)} 
\varepsilon_1 \cdots \hat{\varepsilon}_j \cdots \varepsilon_n \cdot
\biggl( \sum_{k=1}^{n} \frac{\varepsilon_k}{p_k} \biggr)^{s-1} 
\notag \\
& = 
\sum_{j=1}^{n} \frac{1}{p_j} 
\sum_{(\varepsilon_1, \dots, \hat{\varepsilon}_j \dots, \varepsilon_n)} 
\varepsilon_1 \cdots \hat{\varepsilon}_j \cdots \varepsilon_n \cdot
\left\{ 
\biggl( \sum_{\substack{k=1 \\ k \ne j}}^{n} \frac{\varepsilon_j}{p_k} 
+ \frac{1}{p_j} \biggr)^{s-1} +
\biggl( \sum_{\substack{k=1 \\ k \ne j}}^{n} \frac{\varepsilon_j}{p_k} 
- \frac{1}{p_j} \biggr)^{s-1} 
\right\}
\notag \\
& = 
\sum_{j=1}^{n} \frac{1}{p_j} 
\sum_{t=0}^{s-1} \dbinom{s-1}{t} \cdot 
\frac{1+(-1)^{s-1-t}}{p_j^{s-1-t}} \cdot
\sum_{(\varepsilon_1, \dots, \hat{\varepsilon}_j \dots, \varepsilon_n)} 
\varepsilon_1 \cdots \hat{\varepsilon}_j \cdots \varepsilon_n \cdot
\biggl( \sum_{\substack{k=1 \\ k \ne j}}^{n} \frac{\varepsilon_j}{p_k} \biggr)^t,
\end{align}
the equality \eqref{eq:vanishing-sums-appendix-2} is proved by the induction.
Thus we have proved \eqref{eq:vanishing-sums-appendix}.
\end{proof}


\begin{thebibliography}{99}
\bibitem[A]{Andersen-talk}
J. E. Andersen, 
Resurgence analysis of the WRT-TQFT, 
a talk given in the workshop ``A Gauge Summer with BV: Online", June 2020. \\
\url{https://sites.google.com/view/gaugesummerwithbv/home?authuser=0}.


\bibitem[AM]{AP}
J. E. Andersen and W. E. Misteg$\mathring{\rm a}$rd,  
Resurgence Analysis of Quantum Invariants: 
Seifert Manifolds and Surgeries on The Figure Eight Knot,
ArXiv:1811.05376.


\bibitem[A]{A}
D. Auckly,
Topological methods to compute Chern-Simons invariants,
Math. Proc. Cambridge Philos. Soc. {\bf 115} (1994), no. 2, 229--251. 

\bibitem[B]{Beasley09}
C. Beasley, 
Localization for Wilson loops in Chern-Simons theory, 
Adv. Theor. Math. Phys. {\bf 17} (2013), 1--240. 
ArXiv:0911.2687 [hep-th].

\bibitem[BH]{BH}
N. Bleistein and R. A. Handelsman, 
Asymptotic Expansions of Integrals,
Holt, Rinehart and Winston, New York, 1975. 

\bibitem[C]{Chun17}
S. Chun, 
A resurgence analysis of the $SU(2)$ Chern-Simons partition functions 
on a Brieskorn homology sphere $\Sigma(2,5,7)$,
ArXiv:1701.03528 [hep-th].

\bibitem[Ch1]{Chung18}
H.-J. Chung,
BPS invariants for Seifert manifolds, 
{JHEP}, {\bf 2020} (2020).
ArXiv:1811.08863 [hep-th].

\bibitem[Ch2]{Chung20}
H.-J. Chung,
Resurgent Analysis for Some 3-manifold Invariants. 
ArXiv:2008.02786 [hep-th].

\bibitem[Co]{Costin}
O. Costin, 
Asymptotics and Borel Summability, 
Chapman \verb+&+ Hall/CRC Monographs and Surveys 
in Pure and Applied Mathematics, 
Vol. {\bf 141}, CRC Press, Boca Raton, FL, 2009.

\bibitem[CG]{CG11}
O. Costin and S. Garoufalidis, 
Resurgence of the Kontsevich-Zagier power series, 
Ann. Inst. Fourier Grenoble,  {\bf 61} (2011), 1225--1258.
ArXiv:math/0609619 [math.GT].

\bibitem[CCGLS]{CCGLS94}
D. Cooper, M. Culler, H. Gillet, D. D. Long and P. B. Shalen, 
Plane Curves Associated to Character Varieties of 3-Manifolds, 
Invent. Math. {\bf 118} (1994), no 1, 47--84. 


\bibitem[DP]{DP99}
E. Delabaere and F. Pham, 
Resurgent methods in semi-classical asymptotics,
Annales de l'I.H.P. Physique th\'eorique, 
{\bf 71} (1999), 1--94.

\bibitem[Ga]{Garoufalidis}
S. Garoufalidis, 
On the characteristic and deformation varieties of a knot, 
Geom. Topol. Monogr. {\bf 7} (2004), 291--309. 
ArXiv:math/0306230 [math.GT].

\bibitem[Gu]{Gukov04}
S. Gukov, 
Three-Dimensional Quantum Gravity, Chern-Simons Theory, and the A-Polynomial, 
Comm. Math. Phys., {\bf 255} (2005), 577--627.
ArXiv:hep-th/0306165. 

\bibitem[GM]{GM19}
S. Gukov and C. Manolescu,
A two-variable series for knot complements. 
ArXiv:1904.06057 [math.GT].

\bibitem[GMP]{GMP16}
S. Gukov, M. Marin\~o and P. Putrov, 
Resurgence in complex Chern-Simons theory. 
ArXiv:1605.07615 [hep-th].

\bibitem[GPV]{GPV16}
S. Gukov, P. Putrov and C. Vafa, 
Fivebranes and 3-manifold homology, 
JHEP, {\bf 2017} (2017). 
ArXiv:1602.05302 [hep-th]. 
 
\bibitem[GPPV]{GPPV17}
S. Gukov, D. Pei, P. Putrov and C. Vafa, 
BPS spectra and 3-manifold invariants, 
J. Knot Theory Ramif., 
{\bf 29} (2020), No. 02, 2040003. 
ArXiv:1701.06567 [hep-th].


\bibitem[HT1]{Hansen-Takata-1}
S.K. Hansen and T. Takata, 
Reshetikhin-Turaev invariants of Seifert 3-manifolds 
for classical simple Lie algebras,
J. Knot Theory Ramif., 
{\bf 13} (2004), no. 5, 617--668.
ArXiv:math/0209403 [math.GT].


\bibitem[HT2]{Hansen-Takata-2}
S.K. Hansen and T. Takata, 
Quantum invariants of Seifert 3-manifolds and their asymptotic expansions. 
Invariants of knots and 3-manifolds (Kyoto, 2001), 69--87, 
Geom. Topol. Monogr., {\bf 4}, Geom. Topol. Publ., Coventry, 2002. 
ArXiv:math/0210011 [math.GT]

\bibitem[H1]{Hikami-difference}
K. Hikami, 
Difference equation of the colored Jones polynomial for torus knot, 
Int. J. Math. {\bf 15}, 959--965 (2004).
ArXiv:math/0403224.


\bibitem[H2]{Hikami04}
K. Hikami, 
On the quantum invariant for the Brieskorn homology spheres, 
Int. J. Math. {\bf 16} (2005), 661--685. 
ArXiv:math-ph/0405028.


\bibitem[H3]{Hikami04-2}
K. Hikami, 
Quantum invariant, modular form, and lattice points,
{IMRN} {\bf 2005} (2005), Issue 3, 121--154. 
ArXiv:math-ph/0409016.

\bibitem[H4]{Hikami05}
K. Hikami, 
On the Quantum Invariant for the Spherical Seifert Manifold,
{\it Commun. Math. Phys.}, {\bf 268} (2006), 285--319.
arXiv:math-ph/0504082.

\bibitem[H5]{Hikami06}
K. Hikami, 
Quantum invariants, modular forms, and lattice points II, 
J. Math. Phys. {\bf 47} (2006), 102301-32pages.
ArXiv:math/0604091 [math.QA]. 


\bibitem[H6]{Himami11}
K. Hikami, 
Decomposition of Witten-Reshetikhin-Turaev invariant: Linking pairing and modular forms, 
in Chern-Simons Gauge Theory: 20 Years After, 
AMS/IP Stud. Adv. Math. Volume {\bf 50} (2011), 
Amer. Math. Soc., Providence, RI. % 131--151.

\bibitem[HK]{Hikami-Kirillov}
K. Hikami and A. N. Kirillov, 
Torus knot and minimal model, 
Phys. Lett. B, {\bf 575} (2003), 343--348. 
ArXiv:hep-th/0308152.

\bibitem[HM1]{Hikami-Murakami08}
K. Hikami and H. Murakami, 
Colored Jones polynomials with polynomial growth, 
Commun. Contemp. Math., {\bf 10} (2008), 815--834. 
ArXiv:0711.2836 [math.GT].

\bibitem[HM2]{Hikami-Murakami10}
K. Hikami and H. Murakami, 
Representations and the colored Jones polynomial of a torus knot,
in Chern-Simons Gauge Theory: 20 Years After, 
AMS/IP Stud. Adv. Math. Volume {\bf 50} (2011), 
Amer. Math. Soc., Providence, RI. % 153--171.
ArXiv:1001.2680 [math.GT].

% \bibitem[J]{Jeffrey92}
% L. C. Jeffrey, 
% Chern-Simons-Witten invariants of lens spaces
% and torus bundles, and the semiclassical approximation, 
% Commun. Math. Phys. {\bf 147} (1992), 563--604.


\bibitem[K]{Kashaev}
R. M. Kashaev, 
The Hyperbolic Volume of Knots from the Quantum Dilogarithm, 
Letters in Mathematical Physics, {\bf 39} (1997), 269--275.
ArXiv:q-alg/9601025.

\bibitem[KT]{KT99}
R. M. Kashaev and O. Tirkkonen, 
A proof of the volume conjecture on torus knots,
Zap. Nauchn. Sem. S.-Peterburg. Otdel. Mat. Inst. Steklov. (POMI) {\bf 269} (2000), 
Vopr. Kvant. Teor. Polya i Stat. Fiz. {\bf 16}, 262--268, 370; 
translation in J. Math. Sci. (N.Y.) {\bf 115} (2003) 2033--2036.
ArXiv:math/9912210 [math.GT].

% \bibitem[K]{Kirby}
% R. Kirby,   
% The calculus of framed links in $S^3$, 
% Invent. Math. {\bf 45} (1978), 35--56.

\bibitem[KM]{KM91}
R. Kirby and P. Melvin, 
The $3$-manifold invariants of Witten and Reshetikhin-Turaev 
for ${\rm sl}(2,{\bf C})$, 
Invent. Math. {\bf 105} (1991), 473--545.

\bibitem[LR]{LR}
R. Lawrence and L. Rozansky, 
Witten-Reshetikhin-Turaev invariants of Seifert manifolds,
Commun. Math. Phys. {\bf 205} (1999), 287--314.


\bibitem[LZ]{LZ}
R. Lawrence and D. Zagier, 
Modular forms and quantum invariants of 3-manifolds, 
Asian J. of Math. {\bf 3} (1999), 93--108.

\bibitem[Ma]{Marino05}
M. Marin\~o, 
Chern-Simons Theory, Matrix Integrals, and Perturbative 
Three-Manifold Invariants,
Commun. Math. Phys. {\bf 253} (2005), 25--49. 
ArXiv:hep-th/0207096.

\bibitem[Mi1]{Mistegard}
W. E. Misteg$\mathring{\rm a}$rd, 
Quantum Invariants and Chern-Simons Theory, 
PhD Dissertations, Aarhus University,
August 2019. 

\bibitem[Mi2]{Mistegard-talk}
W. E. Misteg$\mathring{\rm a}$rd, 
Quantum Modularity and Resurgence, a talk given in IST Austria, May 2020. \\
\url{https://www.researchgate.net/publication/341574789_Quantum_Modularity_and_Resurgence}


\bibitem[Mo]{Morton95}
H. R. Morton, 
The coloured Jones function and Alexander polynomial for torus knots, 
Math. Proc. Cambridge Philos. Soc., 
{\bf 117} (1995), no. 1, 129--135.


\bibitem[Mu1]{Murakami04}
H. Murakami, 
Asymptotic behaviors of the colored Jones polynomials of a torus knot,
Internat. J. Math. {\bf 15} (2004), 547--555.

\bibitem[Mu2]{Murakami06}
H. Murakami, 
A version of the volume conjecture, 
Adv. Math., {\bf 211} (2007), 678--683.
ArXiv:math/0603217 [math.GT].

\bibitem[MM]{Murakami-Murakami01}
H. Murakami and J. Murakami, 
The colored Jones polynomials and the simplicial volume of a knot, 
Acta Mathematica, {\bf 186} (2001), 85--104. 
ArXiv:math/9905075.

\bibitem[MMOTY]{MMOTY}
H. Murakami, J. Murakami, M. Okamoto, T. Takata and Y. Yokota, 
Kashaev's conjecture and the Chern-Simons invariants of knots and links, 
Experiment. Math. {\bf 11} (2002) 427--435.
ArXiv:math/0203119 [math.GT].

\bibitem[MY]{MY}
H. Murakami and Y. Yokota, 
Volume Conjecture for Knots, 
SpringerBriefs in Mathematical Physics, 
Vol. {\bf 30}, Springer Singapore, 2018.


\bibitem[RT]{RT91}
N. Reshetikhin and V. Turaev,
Invariants of 3-manifolds via link polynomials and quantum groups, 
Invent. Math. {\bf 103} (1991), 547--597.

\bibitem[RJ]{RJ93}
M. Rosso and V. Jones, 
On the invariants of torus knots derived from quantum groups, 
J. Knot Theory Ramification, {\bf 2} (1993), no. 1, 129--135.

\bibitem[S]{Sauzin}
D. Sauzin, 
Introduction to 1-summability and resurgence, 
in Divergent Series, Summability and Resurgence I: 
Monodromy and Resurgence, 
Lecture notes in mathematics {\bf 2153} (2016). 
ArXiv:1405.0356.

\bibitem[W]{Witten89}
E. Witten, 
Quantum field theory and the Jones polynomial, 
Commun. Math. Phys. {\bf 121} (1989), 351.

\bibitem[Z]{Zagier81}
D. Zagier, 
Zetafunktionen und quadratische K\"orper: eine Einf\"uhrung in die 
h\"ohere Zahlentheorie, Hochschultext, Springer-Verlag, 
Berlin-Heidelberg-New York (1981). \\
(Japanese translation: Suuron Nyuumon--zeta-kansuu to nijitai, 
Iwanami Shoten, Tokyo (1990)).

\end{thebibliography}
\end{document}